\documentclass{amsart}

\usepackage[a-2u]{pdfx}

\usepackage[utf8]{inputenc}
\usepackage{lmodern}
\usepackage{amsmath}        % extensions for typesetting of math
\usepackage{amsfonts}       % math fonts
\usepackage{amsthm}         % theorems, definitions, etc.
\usepackage{amssymb} 
\usepackage{bm}             % boldface symbols (\bm)

\usepackage{mathtools}
\usepackage{thmtools}
\usepackage{xcolor}

\theoremstyle{plain}
\newtheorem*{remark}{Remark}
\newtheorem*{definition}{Definition}
\newtheorem{theorem}{Theorem}[section]
\newtheorem{proposition}{Proposition}[section]
\newtheorem{lemma}[theorem]{Lemma}

\makeatletter
\newcommand*{\bu}{\bm u} \newcommand*{\bU}{\bm U} \newcommand*{\bv}{\bm v}
\newcommand{\Bdry}{\partial\Omega}
\newcommand{\rr}{\mathbb{R}}
\newcommand{\norm}[3]{\lVert #1 \rVert_{#2}^{#3}}
\newcommand{\famn}[1]{\{#1\}_{j=1}^N}
\newcommand{\ldvaom}{L^2(\Omega)}
\newcommand{\ldvage}{L^2(\partial \Omega)}
\newcommand{\intom}{\int\limits_{\Omega}}
\newcommand{\diam}{\operatorname{diam}}
\newcommand{\diver}{\operatorname{div}}
\newcommand{\trace}{\operatorname{tr}}
\newcommand{\attr}{\mathcal{A}}
\newcommand{\ddt}{\frac{{\rm d}}{{\rm d}t}}
\newcommand{\calfa}{m_{\alpha}}
\newcommand{\cbeta}{M_{\beta}}
\newcommand{\dimm}[1]{\dim_{#1}^f}

\usepackage{a4wide}

\title[Strong solutions and attractor dimension]{Strong solutions and attractor dimension for 2D NSE\\
with dynamic boundary conditions}

\thanks{The authors thank the Czech Science Foundation, 
Grant Number 20-11027X, for its support.}

\thanks{* Corresponding author.}

\author[D.~Pra\v{z}\'{a}k]{Dalibor Pra\v{z}\'{a}k}
\address{Charles University, Faculty of Mathematics and Physics, Department of Mathematical Analysis, Sokolovsk\'{a}~83, 186~75~Prague~8, Czech~Republic}
\email{\{prazak,zelina\}@karlin.mff.cuni.cz}

\author[M. Zelina]{Michael Zelina$^*$}
\keywords{Navier-Stokes equations, dynamic boundary conditions, 
strong solution, attractor dimension}
\subjclass[2000]{76D05, 35B65, 37L30}

\date{}

\begin{document}

\begin{abstract}
We consider incompressible Navier-Stokes equations in a bounded 2D 
domain, complete with the so-called dynamic slip boundary conditions. Assuming
that the data are regular, we show that weak solutions are strong.
As an application, we provide an explicit upper bound of the fractal
dimension of the global attractor in terms of the physical parameters.
These estimates comply with analogous results in the case of
Dirichlet boundary condition.
\end{abstract}

\maketitle

\section{Introduction}

The 2D incompressible Navier-Stokes equations are an example
of non-linear PDE, for which a rather satisfactory mathematical
theory can be developed. The global existence of a unique weak
solution is available; the solution is smooth if the
data permits. Long-time dynamics can be described by a
finite-dimensional global (or even exponential) attractor.
Its dimension can also be estimated in terms of the problem's parameters. From an extensive bibliography
let us mention the monographs Temam \cite{Te97},
Constantin and Foias \cite{CF:1988}, Robinson \cite{Rob01}.
In particular, the problem of the attractor
dimension is still an area of current research, see Ilyin et al.
\cite{ILZE:21}, \cite{IPZ:2016}.
\par
In the present paper, we aim to extend the analysis to
the case of dynamic slip boundary condition. Here the
usual NS equations are coupled with an evolutionary
problem on the boundary $\partial \Omega$:
\begin{align*}
     \beta \partial_t \bu +  \alpha \bu  + \big[ \bm S \bm n \big]_{\tau} &= \beta \bm{h}
    \\
    \bu \cdot \bm n &= 0
\end{align*}
Here $\bm{S} = \nu \bm{Du}$ is the Cauchy stress, $\nu>0$ the viscosity
of the fluid. Parameter $\alpha>0$ is related to the boundary slip; for
$\alpha=0$ we have perfect slip, while $\alpha \to + \infty$ reduces
to no-slip (zero Dirichlet) condition. The key difference is that the
boundary conditions are not enforced immediately, but only after some
relaxation time $\beta>0$. For the sake of generality, we also include
a boundary force term $\bm{h}$, conveniently multiplied by $\beta$.

%$\color{red}\beta \bm h$ Komentář, že je to pohodlnější.
\par
These problems were extensively studied in \cite{EM-dis}, see also
\cite{Maringova}, where the basic theory of weak solutions was
established, covering a rather general class of relations between
the stress tensor $\bm{S}$ and the shear rate $\bm{Du}$ both of the
polynomial type (Ladyzhenskaya fluid) and even implicit constitutive
relations. Let us also note that existence of finite-dimensional
attractors was established both for 2D and 3D Ladyzhenskaya type
fluid with dynamic slip boundary conditions in \cite{PrZe22}, \cite{PrPr23}.
\par
On the other hand, the problem with the stationary slip condition
(i.e. for $\beta=0$) was studied in \cite{AACG21};
see also \cite{AG-dis}. In particular, the $L^p$ theory for both
weak and strong solutions, as well as the existence of analytic semigroups,
was established for the linear (Stokes) problem.
\par
Our paper is organized as follows: in Section 1, we formulate the problem, and describe the details of analytical setting; in particular,
the function spaces and the weak formulation.
Here we mostly follow \cite{EM-dis}. Section 2 is devoted to
the Stokes system. Key results here consist of deriving the
maximal $W^{2,2}$ regularity, as well as $W^{2,q}$ estimates.
We crucially rely on the (stationary) regularity
results, obtained in \cite{AACG21}.
It appears that the results for $p\neq2$ are not sharp
(maximal), which is perhaps related to the fact that the problem
is not known to generate an analytic semigroup in the $L^p$
setting unless $p=2$. 
\par
In Section 3, we proceed to a non-linear system, including both the
convective term in the interior equations, and a non-linear slip
term on the boundary. We also cover certain class of non-Newtonian
fluids, where the viscosity is bounded, but otherwise depends
on time, space or even the shear rate $|\bm{Du}|$. Section 4
is devoted to estimating the attractor dimension. We use the standard 
method of Lyapunov exponents, focusing on two key steps: differentiability
of the solution semigroup (which relies on the previously obtained
strong regularity), and sharp trace estimates, employing among others
a suitable version of the Lieb-Thirring inequality.  For the reader's convenience, 
several auxiliary results are collected in the Appendix.

\par
Let us briefly mention some further possible research directions.
While the current paper focuses on the case of $\Omega$ bounded, it
would certainly be interesting to also study analogous results
for unbounded domains, regarding both regularity and attractor
dimension; cf. \cite{IPZ:2016} for the case of damped NSE
in $\rr^2$.
Second, a more general class of non-Newtonian fluids could be 
considered, in particular, the Ladyzhenskaya model with growth
exponents $r\neq2$; cf. \cite{KaPr} for the case of Dirichlet
boundary conditions.
\par
Last, but not the least, recall that in case of the 2D Navier-Stokes
equations with homogeneous Dirichlet boundary condition, the
attractor dimension satisfies $\dimm{L^2}{\attr} \le c_0 G$, 
where $G$ is the non-dimensional Grashof number. Our estimates
reduce to that for $\alpha$ large and $\beta$ small as expected.
On the other hand, it is not clear  if the estimate is optimal.
In case of free boundary, an improved estimate
(up to a logarithmic factor)  $\dimm{L^2}{\attr} \le c_0 G^{2/3}$
was shown in \cite{ilyin93}, \cite{ziane98}. It would be interesting to also recover
this as a special case, in the regime where $\alpha$, 
$\beta \to 0$ the estimate is optimal.

\subsection{Problem formulation}

Let $\Omega$ be a bounded Lipschitz domain in $\mathbb{R}^2$. We employ small boldfaced letters to denote vectors and bold capitals for tensors. The symbols ``$\cdot$" and ``$:$" stand for the scalar product of vectors and tensors, respectively. Outward unit normal vector is denoted by $\bm n$ and for any vector-valued function $\bm x : \partial \Omega \rightarrow \mathbb{R}^2$, the symbol $\bm x_\tau$ stands for the projection to the tangent plane, i.e. $\bm x_\tau = \bm x - (\bm x \cdot \bm n) \bm n$. 
\par
Standard differential operators, like gradient ($\nabla$), or divergence ($\text{div}$), are always related to the spatial variables only. By $\bm {Du}$ we understand the symmetric gradient of the velocity field, i.e. $2\bm {Du} = \nabla \bu + (\nabla \bu )^\top$. 
\par
We denote the trace of Sobolev functions as the original function, and if we want to emphasize it, we use the symbol ``tr". Generic constants, that depend just on data, are denoted by $c$ or $C$ and may vary from line to line. 
\par
Our problem is the following. Let $\bm f : (0, T) \times \Omega \rightarrow \mathbb{R}^2 $ and $\bm h : (0, T) \times \partial \Omega \rightarrow \mathbb{R}^2 $  be given external forces and  $\bm u_0 : \overline{\Omega} \rightarrow \mathbb{R}^2$ is the initial velocity. We will also use $\bm F$ as a notation for the whole couple $(\bm f, \bm h)$.  We are looking for the velocity field $\bm u : (0, T) \times \overline\Omega \rightarrow \mathbb{R}^2$ and the pressure $\pi : (0, T) \times \Omega \rightarrow \mathbb{R}$ solutions to the generalized Navier-Stokes system
\begin{align}
\partial_t \bu   - \text{div}\, \bm S ( \bm {D u} ) +  (\bu \cdot \nabla)\bu + \nabla \pi &= \bm f \quad \text{ in } (0,T) \times \Omega , \label{eq:Equation} \\
\text{div}\, \bu &= 0 \quad \text{ in } (0,T) \times \Omega ,  \label{eq:Div}
\end{align}
completed by the boundary and initial conditions
\begin{align}
\beta\partial_t \bu + [\bm s (\bu) + {\bm S ( \bm {D u} } ) \bm{n}]_{\tau} &=  \beta \bm h \quad  \text{ on } (0,T) \times \partial \Omega , \label{eq:Dynamic} \\
\bu \cdot \bm n &= 0 \quad  \text{ on } (0,T) \times \partial \Omega , \label{eq:Impermeability} \\
\bu (0) &= \bm u_0 \quad   \text{ in } \overline{\Omega} , \label{eq:Initial}
\end{align}
where $\beta > 0$ is a fixed number.

By $\bm S : \mathbb{R}^{2\times 2} \rightarrow \mathbb{R}^{2\times 2}$ we understand the viscous part of the Cauchy stress. We require there exists a non-negative potential $U \in \mathcal{C}^1 (\mathbb{R}^+)$ such that $U(0) = 0$ and 
\begin{align}
\bm S (\bm{D}) = \partial_{\bm D} U( |\bm D|^2)  = 2U'( |\bm{D}|^2) \bm{D} . \label{eq:Potential}
\end{align}
Moreover, there hold the coercivity and the growth condition with the power two, i.e. for all symmetrical $2\times 2$ matrices $\bm D$ and $\bm E$ we have the inequalities 
\begin{align}
  (\bm S(\bm D) - \bm S(\bm E)) : ( \bm D - \bm E) & \geq    c_1 |\bm D - \bm E|^2 , \label{eq:Coercivity} \\
  |\partial_{\bm D} U( |\bm D|^2)| =  |\bm S (\bm{D}) | & \leq c_2  |\bm D  | . \label{eq:Growth}
\end{align}
In \autoref{thm:FirstMainTheorem}, we also need higher derivatives of $U$. So, e.g. by $\partial_{\bm D}^2 U( |\bm D|^2)$ we understand
$$\partial_{\bm D}^2 U( |\bm D|^2) = \partial_{\bm D}  ( \partial_{\bm D} U( |\bm D|^2) )= \partial_{\bm D} \bm S (\bm{D}) =2U'( |\bm{D}|^2) \textit{Id} + 4U''( |\bm{D}|^2) \bm{D} \bm{D} .$$

Concerning the boundary term $\bm s$ we work with the similar, but a more general, situation. We consider a differentiable function $\bm s : \mathbb{R}^{2} \rightarrow \mathbb{R}^{2}$ such that $\bm s (\bm 0) = \bm 0$ and for some $s\geq 2$ satisfy for all $\bu, \bv \in \mathbb{R}^2$ the coercivity condition
\begin{align}    
(\bm s(\bu)  - \bm s(\bv) ) \cdot (\bu - \bv) \geq \alpha c_3\left( 1 + |\bu|^{s-2} + |\bv|^{s-2}  \right)|\bu - \bv|^2 ,\label{eq:CoercivityBoundary}
\end{align}
with certain $\alpha > 0$, the growth condition
\begin{align}   
|\bm s(\bu)  - \bm s(\bv) | \leq c_4\left( 1+ |\bu|^{s-2} + |\bv|^{s-2} \right)|\bu - \bv| , \label{eq:GrowthBoundary}
\end{align}
and its first derivative is controlled
\begin{align}   
\bm{s}'(\bu) \bv \cdot \bv \geq \alpha c_5 |  \bv |^2 , \, c_5 \in (0, 1) . \label{eq:DerivativeBoundary}
\end{align}

Typical examples of $\bm S$ satisfying \eqref{eq:Potential}-\eqref{eq:Growth} are 
$$ \bm S (\bm D ) = 2\nu \bm D \text{ and } \bm S (\bm D ) = 2\nu (| \bm D |^2) \bm D, $$
where $\nu$ is either a positive constant or some reasonable shear-dependent function, respectively. The corresponding potentials are
$$ U (|\bm D|^2 )  = \nu |\bm D|^2 \text{ and }   U (|\bm D|^2 ) =  \int\limits_0^{ | \bm D |^2} \nu (s) {\rm d}s\, . $$

\subsection{Main results}

The two main theorems of our article are summarized here. See 
\autoref{ss:FS} for definitions of function spaces.

\begin{theorem}[Strong solutions]\label{thm:FirstMainTheorem}
Let us consider the system \eqref{eq:Equation}--\eqref{eq:DerivativeBoundary} with $\Omega \in \mathcal{C}^{1, 1}$ and the initial condition $\bu_0 \in H$. Concerning the Cauchy stress we further suppose that 
$$U \in \mathcal{C}^3 (\mathbb{R}^+)$$
and 
\begin{align} 
\partial_{\bm D}^2 U( |\bm D|^2) \bm{E} : \bm{E}  =  \partial_{\bm D} \bm S ( \bm{D} ) \bm{E} : \bm{E}  & \geq  c_1 | \bm E |^2 , \label{eq:StressDerivative} \\
|\partial_{\bm D}^2 U( |\bm D|^2) | + |\partial_{\bm D}^3 U( |\bm D|^2) |  & \leq C,
\end{align}
hold for all symmetrical $2\times 2$ matrices $\bm D, \bm E$. Concerning the boundary non-linearity we require that
\begin{align*}
&\bm s \in \mathcal{C}^2(\mathbb{R}^2) \text{ and } \bm s', \bm s'' \text{ are bounded}.
\end{align*}
Let $ 2 < p < +\infty$ be given, we denote
$$t(p) := \frac{ 2p }{p+2}$$
and suppose
\begin{align*}
&\bm F, \partial_t \bm F, \partial_{tt} \bm F \in L^2 (0, T;  V^*), \\
&\bm f  \in L^\infty (0, T; L^p(\Omega)), \, \bm h \in L^\infty (0, T; W^{1-\frac{1}{p}, p}(\partial \Omega)), \\
&\partial_t \bm f \in L^\infty (0, T; L^{t(p)}(\Omega)), \, \partial_t \bm h \in L^\infty (0, T; W^{-\frac{1}{p}, p}(\partial \Omega)).
\end{align*}
Then there is $q > 2$ such that the unique weak solution of \eqref{eq:Equation}--\eqref{eq:DerivativeBoundary} satisfies
\begin{align*}
\bu & \in L^\infty_{\text{loc}} (0, T; W^{2,q}(\Omega) ) , \\
\pi & \in L^\infty_{\text{loc}} (0, T; W^{1,q}(\Omega) ) .
\end{align*}
\end{theorem}

\begin{remark}
The theorem also holds, after some minor modifications, for the case when $\bm S (\bm D) = \bm A (t, x)\bm D$ with symmetrical matrix $\bm A \in \mathcal{C}^2([0, T]; L^\infty (\mathbb{R}^2))$ satisfying the estimate
\begin{align*} 
 c_1 |\bm D |^2 \leq  \bm A (t, x) \bm D : \bm D \leq c_2 |\bm D |^2 , 
\end{align*}
for any symmetrical $2\times 2$ matrix $\bm D$. 
\end{remark}

\begin{remark}
In comparison with the same problem with the Dirichlet boundary condition, see \cite{Kaplicky}, we really need stronger assumptions on the first-time derivative of our data and, moreover, some mild assumption on its second-time derivatives. 
\end{remark}

\begin{theorem}[Dimension estimate]\label{thm:SecondMainTheorem}
Assume both $\bm{S}$ and $\bm{s}$ are linear:
\[
	\bm S (\bm D) = \nu \bm D , \qquad
	\bm s (\bu ) = \alpha \bu ,
\]
and the right-hand side $\bm F = (\bm f, \bm h) \in H$ is independent 
of time, where moreover
\[
	\bm f \in L^p(\Omega),\ \bm h \in W^{1-1/p,p}(\Bdry)
\]
for some $p>2$.
Then the fractal dimension of the global attractor to system \eqref{eq:Equation}--\eqref{eq:Impermeability} satisfies an estimate
\[
\dimm{H}\attr \le c_0 \cdot \frac{ \cbeta }{\calfa^{3/2}} 
\cdot \frac{\ell^2  \norm{\bm F}{H}{} }{\nu^2} \,.
\]
where 
\[
    \ell = \diam \Omega ,
    \qquad
    \calfa = \min\{ 1 , \alpha \ell/ \nu \}, 
    \qquad
    \cbeta = \max\{ 1 , \beta/\ell \} 
\]
and $c_0$ is some non-dimensional constant.
\end{theorem}

\subsection{Function spaces} 	\label{ss:FS}

For a Banach space $X$ over $\mathbb{R}$, its dual is denoted by $X^*$ and $\langle x^*, x \rangle_X$ is the duality pairing. For $p \in [1, \infty]$ we denote $(L^p(\Omega), || \cdot ||_{L^p(\Omega)})$ and $(W^{1,p} (\Omega), || \cdot ||_{W^{1,p}(\Omega)})$ the Lebesgue and Sobolev spaces with corresponding norms. We often write just $|| \cdot ||_p$ or $|| \cdot ||_{1,p}$. The space of functions $\bu:[0,T]\to X$ which are $L^p$ integrable or (weakly) continuous with respect to time is denoted by $L^p(0,T; X)$,  $\mathcal{C}([0,T]; X)$  or $\mathcal{C}_{\rm w}([0,T]; X)$ respectively.

Because of the presence of the time derivative on the boundary, we need to pay close attention to the boundary terms. Thus, we need more refined function spaces. We will follow the notation of \cite{Maringova}. We introduce the spaces 
\begin{align*}
\mathcal{V} &:= \{ ( \bu , \bm g) \in \mathcal{C}^{0,1} (\overline{\Omega}) \times \mathcal{C}^{0,1}  (\partial \Omega); \text{div}\, \bu = 0 \text{ in } \Omega, \bu \cdot \bm n = 0 \text{ and } \bu = \bm g \text{ on } \partial \Omega    \} , \\
V &:= \operatorname{cl}(\mathcal{V},V), \text{ where }
|| ( \bu , \bm g) ||_{V} := || \bm {Du} ||_{L^2(\Omega)}  + \alpha ||\bm g ||_{L^2(\partial \Omega)} ,  \\
H &:= \operatorname{cl}(\mathcal{V},H), \text{ where }
|| ( \bu , \bm g) ||_H^2 := || \bu ||_{L^2( \Omega)}^2 + \beta || \bm g ||_{L^2(\partial \Omega)}^2.
%V &:= \overline{\mathcal{V}}^{|| \cdot ||_{V}}, \text{ where } || ( \bu , \bm g) ||_{V} := || \bm {Du} ||_{L^2(\Omega)}  + \alpha ||\bm g ||_{L^2(\partial \Omega)} ,\\
%H &:= \overline{V}^{|| \cdot ||_H}, \text{ where } 
%|| ( \bu , \bm g) ||_H^2 := || \bu ||_{L^2( \Omega)}^2 + \beta || \bm g ||_{L^2(\partial \Omega)}^2.
\end{align*}

Space $V$ is both reflexive and separable. Observe that thanks to Korn's inequality (see \autoref{thm:Korn} in Appendix), the norm in $V$ is equivalent to the standard $W^{1,2}$ norm. Next, $H$ is Hilbert space identified with its own dual $H^*$, endowed with the inner product
$$ ( (\tilde{\bu}, \tilde{\bm g}), ( \bu, \bm g))_H := \int\limits_\Omega \tilde{\bu} \cdot  \bu  \, {\rm d}x + \beta \int\limits_{\partial \Omega} \tilde{\bm g}  \cdot  \bm g \, {\rm d}S . $$ 

The duality pairing between $V$ and $V^*$ is defined in a standard way as a continuous extension of the inner product $(\cdot, \cdot)_H$ on $H$. Note that there is a Gelfand triplet
$$ V \hookrightarrow H \equiv H^* \hookrightarrow V^* ,$$
where both embeddings are continuous and dense.
\par
It will be useful (and certainly is of independent interest) to have
intrinsic description of the above spaces. Let us denote
\begin{align*}
	W^{1,p}_{\sigma,\bm n}(\Omega) &= \{ \bu \in W^{1,p}(\Omega);
	\ \textrm{ $\diver \bu =0$ in $\Omega$,
                    $\bu\cdot \bm{n} = 0$ on $\Bdry$} \},
\\
	L^p_{\sigma,\bm n} &= \operatorname{cl}(\mathcal V, L^p(\Omega \times \Bdry))
	\, \textrm{with }\norm{(\bu,\bm g)}{L^p(\Omega \times \Bdry)}{}
		= \norm{\bu}{L^p(\Omega)}{} + \norm{\bm g}{L^p(\Bdry)}{},
\\
	L^p_{\sigma,\bm n}(\Omega) &= \{ \bu \in L^p(\Omega);\ 
		\textrm{ $\diver \bu =0$ in $\Omega$, 
					$\bu\cdot \bm{n} = 0$ on $\Bdry$} \},
\\
	L^p_{\tau}(\Bdry) &= \{ \bm g \in L^p(\Bdry);\ 
	\bm g \cdot \bm{n} = 0 \}.
\end{align*}
Note that $L^2_{\sigma,\bm n} = H$; and the normal trace in this space
is well-defined, cf. \cite[Section 10.3.]{FN09}. Now, 
it is not difficult to see (by an argument similar to the lemma below) 
that
\[
	V = \{ (\bu,\trace{\bu})\,; \bu \in W^{1,2}_{\sigma,\bm n}(\Omega) \}.
\]
Furthermore, if $\rho\ge 1$ is such that 
$\trace : W^{1,2}(\Omega) \to L^\rho(\Bdry)$, then
$V \hookrightarrow W^{1,p}_{\sigma,\bm n}(\Omega) \times
L^\rho_{\tau}(\Bdry)$, and hence also $( W^{1,p}_{\sigma,\bm n}(\Omega))^* 
\times L^{\rho'}_{\tau}(\Bdry) \hookrightarrow V^*$.
Finally, we claim that in the class of $L^p$ functions, the
interior and boundary values decouple as well.
\begin{lemma}	\label{lem:Decouple}
Let $\Omega \subset \rr^d$ be a bounded $\mathcal{C}^{1,1}$ domain. Then
for any $p\in(1,+\infty)$ one has
\[
	L^p_{\sigma,\bm n} 
		= L^p_{\sigma,\bm n}(\Omega) \times L^p_\tau(\Bdry) \,.
\]
\end{lemma}

\begin{proof}
First inclusion $\subset$ is obvious. To prove the second one,
we will establish that
\begin{align}	\label{con1}
	L^p_{\sigma,\bm n}(\Omega) \times \{0\} &\subset L^p_{\sigma,\bm n} ,
\\				\label{con2}
	\{0\} \times  L^p_{\tau}(\Bdry) & \subset L^p_{\sigma,\bm n} .
\end{align}
Inclusion \eqref{con1} is also clear since the space 
\[
	\mathcal{D}(\Omega) = \{ \bu \in \mathcal{C}^{\infty}_0(\Omega);\ 
				\diver \bu = 0 \}
\]
is dense in $L^p_{\sigma,\bm n}(\Omega)$, see e.g. \cite[Theorem III.2.3]{Galdi11}.
It remains to prove \eqref{con2}, i.e. for a given $\bm g \in L^p_{\tau}(\Bdry)$ we need to find smooth extension $\bu$ such that both $\norm{\bu - \bm g}{L^p(\Bdry)}{}$ and $\norm{\bu}{L^p(\Omega)}{}$ are small.
\par
Since $\Bdry$ is regular, we can assume that $\bm g$ is $\mathcal{C}^1$. Let $\bu^{(1)}$ be its smooth extension such that $\norm{\bu^{(1)}}{L^p(\Omega)}{} < \varepsilon$. To ensure the solenoidality, we finally set
\[
	\bu = \bu^{(1)} - \bu^{(2)}, \quad \textrm{where }
		\bu^{(2)} = \mathcal{B}[\diver \bu^{(1)}],
\]
$\mathcal B$ being the Bogovskii operator; see \cite[Section 10.5]{FN09}
for details. In particular, since $\bu^{(1)}\cdot \bm n 
= \bm g^{(1)} \cdot \bm n = 0$ on $\Bdry$, we have 
$\int_\Omega \diver \bu^{(1)} \, {\rm d}x =0$. It follows from \cite[Theorem 10.11]{FN09} that $\bu^{(2)} \in W^{1,p}_0(\Omega)$ and
\[
\norm{\bu^{(2)}}{L^p(\Omega)}{} 
	\le C \norm{\bu^{(1)}}{L^p(\Omega)}{} \le C \varepsilon ,
\]
where $C>0$ only depends on $\Omega$ and $p$.
Hence $\bu$ is the sought-for interior extension to $\bm g$.
\qed \end{proof}

\begin{remark}
It is worth noting that we will not actually need a full strength of \autoref{lem:Decouple}; rather just a very special case. Let us consider a couple $(\bu, \bm g)$ such that $\bu \in L^2_{\sigma,\bm n}(\Omega)$ and $\bm g$ is a trace of $\bv \in W^{1,2}_{\sigma,\bm n}(\Omega)$. Then $(\bv, \bm g) \in V$, which we observed above, and $(\bu - \bv, \bm 0) \in H$ by \eqref{con1}. Therefore  $(\bu, \bm g) \in H$. It works similarly also for  $\bu \in (W^{1,2}_{\sigma,\bm n}(\Omega))^*$, we would obtain $(\bu, \bm g) \in V^*$. Let us also remark that the whole argument needs $\Omega$ to be just a Lipschitz domain.
\end{remark}

%%%% vyhodit nize uvedene %%%
\iffalse
\begin{remark}
Due to \eqref{eq:CoercivityBoundary} we could also work with the norm  $|| ( \bu , \bm g) ||_{V_s} := || \bm {Du} ||_{L^2(\Omega)}  + \alpha ||\bm g ||_{L^s(\partial \Omega)}$. In our work, it is not really important.
\end{remark}
Next, we need to specify how to create the duality pairing for objects defined only inside of $\Omega$. So, if $\bm f \in \left(W^{1, 2}_{\bm n} (\Omega)\right)^*$, we identify it with a couple $(\bm f , \bm 0) \in V^*$ and we observe $\langle \bm f, \bm \varphi \rangle_{V} = \langle \bm f, \bm \varphi \rangle_{W^{1, 2}_{\bm n} (\Omega)}$. In particular, if $\bm f \in L^2(\Omega)$ then the definiton means that $\langle \bm f, \bm \varphi \rangle_{L^2(\Omega)}  = \int\limits_{\Omega} \bm f \cdot \bm \varphi \, {\rm d}x$.
Conversely, let us consider the case when $\bm h \in L^2(\partial \Omega)$ is given. Can we find $\bm f$ such that $(\bm f, \bm h) \in H$? Of course, if $\bm h = \text{tr}\, \bm f$ for some $\bm f \in W^{1, 2}(\Omega)$, then $(\bm f, \bm h) \in V$ by the definition and thus also $(\bm f, \bm h) \in H$. The general situation is not so obvious, but even in this case, we can find $\bm f$ such that $(\bm f, \bm h) \in H$. 
\fi

\subsection{Weak formulation}

Here, we formally derive the proper notion of a weak solution. We take a scalar product of \eqref{eq:Equation} with the smooth test function $\bm \varphi \in \mathcal V$, integrate over the whole $\Omega$ and use Gauss's theorem to get
\begin{align*}
\int\limits_{\Omega} \partial_t \bu \cdot \bm \varphi + & \int\limits_{\Omega} (\bu \cdot \nabla ) \bu \cdot \bm \varphi + \int\limits_{\Omega} \bm S ( \bm{Du})  : \nabla \bm{\varphi} - \int\limits_{\partial \Omega} \left[\bm S ( \bm{Du}) \bm n \right]_\tau   \cdot \bm{\varphi}  \\
& \qquad = \int\limits_{\Omega} \bm f \cdot \bm \varphi  - \int\limits_{\Omega} \pi \, \text{div}\, \bm \varphi + \int\limits_{\partial \Omega} \pi \, \bm n \cdot \bm \varphi  .
\end{align*}
The pressure terms vanish due to $\text{div}\, \bm \varphi = 0$. Similarly, the tangential projection of boundary terms can be dropped as $\bm \varphi \cdot \bm n = 0$ on $\Bdry$; we follow this convention from now on.
Together with symmetricity of $\bm S ( \bm{Du})$ we obtain
\begin{align*}
\int\limits_{\Omega} \partial_t \bu \cdot \bm \varphi  + \int\limits_{\Omega} \bm S ( \bm{Du})  : \bm{D \varphi} - \int\limits_{\partial \Omega} \left[\bm S ( \bm{Du}) \bm n \right]_\tau   \cdot \bm{\varphi}   = \int\limits_{\Omega} \bm f \cdot \bm \varphi  - \int\limits_{\Omega} (\bu \cdot \nabla ) \bu \cdot \bm \varphi .
\end{align*}
Next, we use \eqref{eq:Dynamic} to finally get
\begin{align*}
\int\limits_{\Omega} \partial_t \bu \cdot \bm \varphi & + \beta  \int\limits_{\partial \Omega} \partial_t \bu \cdot \bm \varphi + \int\limits_{\Omega} \bm S ( \bm{Du})  : \bm{D \varphi} + \int\limits_{\partial \Omega} \bm s(\bu) \cdot \bm \varphi \\
& = \int\limits_{\Omega} \bm f \cdot \bm \varphi + \beta \int\limits_{\partial \Omega} \bm h \cdot \bm \varphi - \int\limits_{\Omega} (\bu \cdot \nabla ) \bu \cdot \bm \varphi , 
\end{align*}
which we rewrite as
\begin{align*}
( \partial_t \bu , \bm \varphi )_H + & \int\limits_{\Omega} \bm S ( \bm{Du})  : \bm{D \varphi} + \int\limits_{\partial \Omega} \bm s(\bu) \cdot \bm \varphi  = ( \bm F , \bm \varphi )_H - \int\limits_{\Omega} (\bu \cdot \nabla ) \bu \cdot \bm \varphi .
\end{align*}

Of course, rigorously, the scalar product must be replaced by the duality pairing. From this point, it is not difficult to realize that we are able to get the usual apriori estimates for $\bu$ and $\partial_t \bu$. Hence, we introduce the following definition.

\begin{definition}
By a weak solution of \eqref{eq:Equation}--\eqref{eq:Initial} we understand the function 
\begin{align*}
\bu & \in L^2(0, T; V) \cap \mathcal{C}([0, T]; H) \text{ and } \\
\partial_t \bu & \in L^{2}(0, T; V^*) 
\end{align*}
that for a.e. $t \in (0, T)$ and any $\bm \varphi \in V$ satisfies the identity
\begin{align} \label{eq:WF}
\langle \partial_t \bu, \bm \varphi \rangle + \int\limits_{\Omega} \bm S ( \bm{Du})  : \bm{D \varphi} +  \int\limits_{\partial \Omega} \bm s (\bu) \cdot \bm \varphi = \langle \bm F, \bm \varphi \rangle - \int\limits_{\Omega} (\bu \cdot \nabla)\bu \cdot \bm \varphi,
\end{align}
the initial condition $\bu(0)=\bu_0$ holds in $H$,
%$$ \lim_{t \rightarrow 0_+} || \bu (t) - \bu_0||_H = 0 $$
and for all $t \in [0, T]$ it satisfies the energy equality
$$ \frac{1}{2} || \bu (t) ||_H^2 +  \int\limits_0^t \int\limits_{\Omega} \bm S ( \bm{Du})  : \bm{Du} + \int\limits_0^t \int\limits_{\partial \Omega}  \bm s (\bu) \cdot \bu  =  \frac{1}{2} || \bu_0 ||_H^2 + \int\limits_0^t \langle \bm F, \bu \rangle . $$
\end{definition}

\subsection{Dynamical systems}

We recall some basic notions from the theory of dynamical systems. 
Let $\mathcal{X}$ be (a closed subset to) a normed space. Family of mappings $\{ \Sigma_t \}_{t \geq 0} : \mathcal{X} \rightarrow \mathcal{X}$ is called a semigroup provided that $\Sigma_0 = I$ and $ \Sigma_{t+s} = \Sigma_t \Sigma_s$  for all
$s$, $t\ge0$. Requiring also continuity of the map $(t,x)\mapsto \Sigma_t x$,
the couple $(\Sigma_t, \mathcal{X})$ is referred to as a dynamical system. 
\par
Set $\mathcal{A} \subset \mathcal{X}$ is called a global attractor to the dynamical system $(\Sigma_t, \mathcal{X})$ if 
\begin{itemize}
\item[(i)] $\mathcal{A}$ is compact in $\mathcal{X}$,
\item[(ii)] $\Sigma_t \mathcal{A} = \mathcal{A}$ for all $t \geq  0$ and
\item[(iii)] for any bounded $\mathcal{B} \subset \mathcal{X}$ there holds 
$$ \text{dist} (\Sigma_t \mathcal{B}, \mathcal{A}) \rightarrow 0 \text{ as } t \rightarrow \infty,  $$ 
where $\text{dist} (\mathcal{B}, \mathcal{A})$ is the standard Hausdorff semi-distance of the set $\mathcal{B}$ from the set $\mathcal{A}$, defined as $\text{dist} (\mathcal{B}, \mathcal{A}) = \sup_{a \in \mathcal{A}} \inf_{b\in \mathcal{B}} || b - a||_ \mathcal{X}$.
\end{itemize}
Let us note that a dynamical system can have at most one global attractor. The condition (ii) says that the global attractor is (fully) invariant with respect to $\Sigma_t$.
\par
Fractal dimension of a compact set $\mathcal{K} \subset \mathcal{X}$ is defined by
$$ \dimm{\mathcal{X}} \mathcal{K} := \limsup_{\varepsilon \rightarrow 0_+} \frac{\log N_\varepsilon^{\mathcal{X}} (\mathcal{K})}{-\log \varepsilon} \, , $$
where $N_\varepsilon^{\mathcal{X}} (\mathcal{K})$  denotes the minimal number of $\varepsilon$-balls needed to cover the set $\mathcal{K}$. See e.g.
\cite{Rob11} for further properties as well as related results.

\section{Stokes system}

Let us start with the basic properties of the Stokes operator, corresponding
to the dynamic boundary conditions. Here we mostly follow the results
of \cite{Maringova}, \cite{BMR2007} as well as  \cite{AG-dis}, \cite{AACG21}.

\subsection{Eigenvalue problem - ON basis}

\begin{theorem}[Basis of $V$]\label{thm:Basis}
There exists the sequence $\{ \bm \omega_k \}_{k\in \mathbb{N}}$ which is a basis in both $V$ and $H$, it is orthogonal in $V$ and orthonormal in $H$. Further, there is a non-decreasing sequence $\{\mu_k\}_{k\in \mathbb{N}}$ with $\lim_{k \rightarrow +\infty} \mu_k = +\infty$. For every $k \in \mathbb{N}$ the function $\bm \omega_k$ solves the problem
\begin{align}
- \text{div}\, \bm{D \omega}_k + \nabla \pi &=  \mu_k \bm \omega_k \quad \text{ in } \Omega   \label{eq:Basis1},\\
\text{div}\, \bm \omega_k  &= 0  \quad  \text{ in }  \Omega  \label{eq:Basis2},\\
 \alpha \bm \omega_k + [(\bm{D\omega}_k) \bm n]_\tau  &= \mu_k \beta \bm \omega_k \quad  \text{ on } \partial \Omega \label{eq:Basis3},\\
 \bm \omega_k \cdot \bm n &= 0 \quad  \text{ on } \partial \Omega \label{eq:Basis4}
\end{align} 
in the weak sense. Equivalently, the equations can be written as
\begin{align}
(\bm \omega_k, \bm \varphi)_{V} = \mu_k (\bm \omega_k, \bm \varphi)_H ,\, \forall \bm \varphi \in V \label{eq:VlastniCisla}.
\end{align}
Moreover, for $P^N$, a projection of $V$ to the linear hull of $\{ \bm \omega_k \}_{k = 1}^N$ defined by
$$ P^N \bu := \sum_{k = 1}^N (\bu, \bm \omega_k)_H \bm \omega_k , $$
it holds that for any $\bu \in V$
\begin{align*}
||P^N \bu ||_H & \leq || \bu ||_H, \\
||P^N \bu ||_V & \leq || \bu ||_V, \\
P^N \bu & \rightarrow \bu \text{ in } V \text{ as } N \rightarrow +\infty .
\end{align*}
\end{theorem}

\begin{proof}
See \cite{Maringova} or \cite{EM-dis}.
\qed \end{proof}

\subsection{Stokes problem - stationary}

Let us consider the following system
\begin{align}
- \text{div}\, \bm{D u} + \nabla \pi &=  \bm f \quad \text{ in }  \Omega \label{eq:StokesStationary1} ,\\
\text{div}\, \bu  &= 0  \quad  \text{ in } \Omega \label{eq:StokesStationary2},\\
 \alpha \bu + [(\bm{Du}) \bm n]_\tau &= \bm h \quad  \text{ on } \partial \Omega \label{eq:StokesStationary3},\\
 \bu \cdot \bm n &= 0 \quad  \text{ on } \partial \Omega \label{eq:StokesStationary4}.
\end{align}
It was examined in \cite{AG-dis}, \cite{AACG21} in the three-dimensional case. Here we formulate the analogue two-dimensional results.

\begin{theorem}[Existence in $W^{2,2}$ - stationary Stokes]\label{thm:StokesStationarHilbert}
Let $\alpha > 0$, $\Omega \in \mathcal{C}^{1,1}$ and
\begin{align*}
\bm f \in L^2(\Omega),\, \bm h \in W^{\frac{1}{2}, 2}(\partial \Omega) . 
\end{align*}
Then the problem \eqref{eq:StokesStationary1}--\eqref{eq:StokesStationary4} has a unique solution $(\bu, \pi ) \in W^{2, 2}(\Omega) \times W^{1, 2}(\Omega)$ satisfying
\begin{align*}
|| \bu ||_{2,2} + || \pi ||_{1,2} \leq C(\Omega)\left( 1 + \frac{1}{\min \{ 2, \alpha \}} \right) \left( || \bm f ||_2 + || \bm h ||_{\frac{1}{2}, 2} \right) .
\end{align*}
\end{theorem}

\begin{proof}
See Corollary 2.4.5 in \cite{AG-dis}.
\qed \end{proof}

\begin{theorem}[Existence in $W^{1,p}$ - stationary Stokes]\label{thm:StokesStationarWeak}
Let $\alpha > 0$, $p \in (1, +\infty )$, $\Omega \in \mathcal{C}^{1,1}$ and 
\begin{align*}
\bm f \in L^{ t(p)} (\Omega),\, \bm h \in W^{-\frac{1}{p}, p}(\partial \Omega) \, \text{with } t(p) = \frac{ 2p }{p+2} \, .
\end{align*}
Then the unique solution of \eqref{eq:StokesStationary1}--\eqref{eq:StokesStationary4} belongs to $(\bu, \pi ) \in W^{1, p}(\Omega) \times L^p(\Omega)$ and satisfies
\begin{align*}
|| \bu ||_{1,p} + || \pi ||_{p} \leq C(\Omega, p, \alpha) \left( || \bm f ||_{t(p)} + || \bm h ||_{-\frac{1}{p}, p} \right) .
\end{align*}

\end{theorem}

\begin{proof}
See Corollary 2.5.6 in \cite{AG-dis}.
\qed \end{proof}

\begin{theorem}[Existence in $W^{2,p}$ - stationary Stokes]\label{thm:StokesStationarStrong}
Let $\alpha > 0$, $p \in (1, +\infty )$, $\Omega \in \mathcal{C}^{1,1}$ and
\begin{align*}
\bm f \in L^{p}(\Omega),\, \bm h \in W^{1-\frac{1}{p}, p}(\partial \Omega) . 
\end{align*}
Then the unique solution of \eqref{eq:StokesStationary1}--\eqref{eq:StokesStationary4} belongs to $(\bu, \pi ) \in W^{2, p}(\Omega) \times W^{1, p}(\Omega)$ and satisfies
\begin{align*}
|| \bu ||_{2,p} + || \pi ||_{1,p} \leq C(\Omega, p, \alpha) \left( || \bm f ||_p + || \bm h ||_{1-\frac{1}{p}, p} \right) .
\end{align*}

\end{theorem}

\begin{proof}
See Theorem 2.5.9 in \cite{AG-dis}.
\qed \end{proof}

\begin{remark}
Due to \cite[Remark 2.6.16]{AG-dis} previous three theorems also hold with the leading elliptic term in the form $$- \diver (\bm A(x) \nabla) \bu .$$ 
\end{remark}

\subsection{Stokes problem -- evolutionary}

The evolutionary version of the previous system looks like this
\begin{align}
\partial_t \bu - \text{div}\, \bm{D u} + \nabla \pi &=  \bm f \quad \text{ in } (0, T) \times \Omega \label{eq:StokesEvolutionary1},\\
\text{div}\, \bu  &= 0  \quad  \text{ in } (0, T) \times \Omega \label{eq:StokesEvolutionary2},\\
\beta \partial_t \bu  + \alpha \bu + [(\bm{Du}) \bm n]_\tau &= \beta \bm h \quad  \text{ on } (0, T) \times \partial \Omega \label{eq:StokesEvolutionary3},\\
 \bu \cdot \bm n &= 0 \quad  \text{ on } (0, T) \times \partial \Omega  \label{eq:StokesEvolutionary4}, \\
 \bu (0) & = \bu_0 \quad  \text{ in } \overline{ \Omega  } \label{eq:StokesEvolutionary5}. 
\end{align}
Here, we will assume that $\alpha$, $\beta > 0$. The first result is then the following.

\begin{theorem} \label{thm:StokesEvolutionarStrong}
Let $\Omega \in \mathcal{C}^{0,1}$ and
\begin{align*}
\bm F & \in L^2(0, T; V^* ),  \\
\bu_0 & \in H.
\end{align*}
Then the problem \eqref{eq:StokesEvolutionary1}--\eqref{eq:StokesEvolutionary5} has the unique weak solution $(\bu, \pi)$ and the velocity $\bu$ satisfies
\begin{align*}
\bu  & \in L^\infty(0, T; H) \cap L^2(0, T; V).
\end{align*}
\begin{itemize}
\item[(i)] Suppose further that
\begin{align*}
\partial_t \bm F   & \in L^2(0, T; V^* ).  
\end{align*}
Then there also holds
\begin{align*}
\partial_t \bu & \in L^\infty_{\text{loc}}(0, T; H) \cap L^2_{\text{loc}}(0, T; V), \\
\bu  & \in L^\infty_{\text{loc}}(0, T; V) .
\end{align*}
Moreover, if $\bm f(0) \in L^2(\Omega), \bm h(0) \in W^{\frac{1}{2}, 2}(\partial \Omega)$ and $\bu_0 \in V \cap W^{2,2}(\Omega)$, then the previous result holds globally in time.
\item[(ii)] Alternatively, let
\begin{align*}
\bm F & \in L^2(0, T; H ). 
\end{align*}
Then the solution satisfies
\begin{align*}
\bu  & \in L^\infty_{\text{loc}}(0, T; V) \cap L^2_{\text{loc}}(0, T; W^{1, 4}(\Omega)), \\
\partial_t \bu & \in L^2_{\text{loc}}(0, T; H).
\end{align*}
\end{itemize}

\end{theorem}

\begin{proof}
The starting point is the Galerkin approximation, i.e. for a given $n \in \mathbb{N}$ we look for the solution in the form 
\begin{align*}
\bu^n & = \sum_{k=1}^n c_k^n(t) \bm \omega_k ,
\end{align*}
where $c_k^n$ are some functions of time satisfying, for all $k = 1, \dots , n$, the system
\begin{align}
(\partial_t \bu^n, \bm \omega_k)_H + \int\limits_{\Omega} \bm {Du}^n  : \bm{D\omega}_k + \alpha \int\limits_{\partial \Omega} \bu^n \cdot \bm \omega_k = \left< \bm F, \bm \omega_k \right>  \label{eq:GalerkinStokes}
\end{align}
together with the initial condition 
\begin{align*}
\bu^n(0) &= \bu_0^n,
\end{align*}
where $\bu_0^n$ is the orthogonal projection of $\bu_0$ on the space spanned by $\{ \bm \omega_k \}_{k=1}^n$. This can also be written as $c_k^n (0)= ( \bu_0, \bm \omega_k )_H$. The existence of these functions $c_k^n$ follows from the standard theory. 

Existence of the solution is done in a standard way. We multiply 
\eqref{eq:GalerkinStokes} by $c_k^n(t)$ and sum the result over $k = 1, \dots , n$ to obtain
\begin{align*}
\frac{1}{2} \cdot \frac{{\rm d}}{{\rm d}t} ||\bu^n ||_H^2  + || \bu^n ||_V^2= \left< \bm F, \bu^n \right>.
\end{align*}
By Young's inequality, we get the uniform estimate for $\bu^n$ in the form
\begin{align*}
\bu^n &\in L^\infty (0, T; H) \cap L^2 (0, T; V) .
\end{align*}
Next, using the duality argument we also obtain that the time derivative 
is bounded in 
\begin{align*}
\partial_t \bu^n   & \in L^{2} (0, T; V^*).
\end{align*}
Passing to the limit is straightforward and uniqueness is standard.

\underline{Proof of (i).} Because of the linearity of our system it is clear that the function $\bv := \partial_t \bu$ satisfies the same system as $\bu$, just with $\partial_t \bm f$, $\partial_t \bm h$ instead of $\bm f$, $\bm h$. Rigorously, we can take the time derivative of \eqref{eq:GalerkinStokes} and multiply the result by $(c_k^n)'(t)$ and sum over all indices. We will obtain the uniform estimate 
\begin{align*}
\partial_t \bu^n \in L^\infty_{\text{loc}}(0, T; H) \cap L^2_{\text{loc}}(0, T; V)  .
\end{align*}
Of course, the result will hold only locally in time, because we do not prescribe any condition on $(c_k^n)'(0)$. It means that we need to verify that $\partial_t \bu^n(t_0) \in H$ for some $t_0 \in [0, T]$. This can be done if we multiply \eqref{eq:GalerkinStokes} by $(c_k^n)'$. Let us remark that if $\bu_0, \bm f(0), \bm h(0)$ would be better we would obtain the global result. See also Theorem III.3.5 in \cite{Temam} in the Dirichlet setting.

Finally, the fact that both $\bu$ and $\partial_t \bu$ are in $L^2_{\text{loc}}(0, T; V)$ implies that $\bu \in L^\infty_{\text{loc}}(0, T; V)$.

\underline{Proof of (ii).} First, we multiply \eqref{eq:GalerkinStokes} by $(c_k^n)'(t)$ and sum over $k$'s to obtain
\begin{align*}
|| \partial_t \bu^n ||_H^2 + \frac{1}{2} \cdot \frac{{\rm d}}{{\rm d}t} || \bu^n||_V^2  = \left< \bm F, \partial_t \bu^n \right> . 
\end{align*}
Second, if we multiply \eqref{eq:GalerkinStokes} by $\mu_k c_k^n(t)$ and sum again, we get
\begin{align*}
\frac{1}{2} \cdot \frac{{\rm d}}{{\rm d}t} || \bu^n ||_{V}^2 + (\bu^n, L^n)_V = \left<\bm F, L^n\right> , 
\end{align*}
where 
\begin{align*}
L^n := \sum_{k=1}^n \mu_k c_k^n(t) \bm \omega_k .
\end{align*}
Let us note that we used the following identity
\begin{align*}
\sum_{k=1}^n (\partial_t \bu^n, \mu_k c_k^n(t) \bm \omega_k )_H &= \sum_{k=1}^n c_k^n(t) \left[ \int\limits_{\Omega} \partial_t \bu^n \mu_k \bm \omega_k + \beta \int\limits_{\partial \Omega} \partial_t \bu^n  \mu_k \bm \omega_k \right] \\
&= \sum_{k=1}^n c_k^n(t) (\partial_t \bu^n, \bm \omega_k)_{V} = (\partial_t \bu^n, \bu^n)_{V} \\
& = \frac{1}{2} \cdot \frac{{\rm d}}{{\rm d}t} || \bu^n ||_{V}^2.
\end{align*}
Next, we add both equations to obtain
\begin{align*}
|| \partial_t \bu^n ||_H^2 & + \frac{{\rm d}}{{\rm d}t} || \bu^n ||_{V}^2  +   (\bu^n, L^n)_V    = \left< \bm F, \partial_t \bu^n + L^n \right> .
\end{align*}
Observe that $L^n \in V$, and so, by \eqref{eq:VlastniCisla}, we obtain
\begin{align*}
(L^n, L^n)_H  = || L^n ||_H^2 = (\bu^n, L^n)_{V} ,
\end{align*}
and therefore
\begin{align*}
|| \partial_t \bu^n ||_H^2 & + \frac{{\rm d}}{{\rm d}t} || \bu^n ||_{V}^2  +  || L^n ||_H^2   = \left< \bm F, \partial_t \bu^n + L^n \right> .
\end{align*}

Now, let us choose some small $t_0 \in (0, T)$ for which $\bu^n(t_0) \in V$. We integrate the relation over $(t_0, T)$ and use Hölder's and Young's inequalities to obtain
\begin{align*}
\int\limits_{t_0}^t & || \partial_t \bu^n ||_H^2  +  2 || \bu^n (t) ||_{V}^2  +  \int\limits_{t_0}^t  || L^n ||_H^2   \leq  2|| \bu^n (t_0) ||_{V}^2 +   \int \limits_{t_0}^t  || \bm F ||_H^2   .  
\end{align*}
On the right-hand side we can estimate all terms, and therefore, we get the following uniform estimates
\begin{align*}
\partial_t \bu^n & \in  L^2_{\text{loc}}(0, T; H), \\
\bu^n & \in L^\infty_{\text{loc}}(0, T; V) , \\
L^n & \in L^2_{\text{loc}}(0, T; H) .
\end{align*}

It remains to show that the last property gives us the estimate of $\bu^n$ in $L^2_{\text{loc}}(0, T; W^{1, 4}(\Omega))$. Because any $\bm \omega_k$ solves \eqref{eq:Basis1}--\eqref{eq:Basis4} we can apply \autoref{thm:StokesStationarWeak} for $p =4$, $\alpha > 0$ and $(\bm f, \bm h) = (\sum_k \mu_k c_k^n \bm \omega_k, \sum_k \mu_k c_k^n \text{tr}\, \bm \omega_k)$. We obtain that $\sum_k c_k^n\bm \omega_k$ belongs into $W^{1,4}(\Omega)$, more specifically, there holds
\begin{align*}
 \bigl|\bigl|\sum_k c_k^n\bm \omega_k \bigr|\bigr|_{1,4} &\leq C (\Omega, \alpha) \left( \bigl|\bigl| \sum_k \mu_k c_k^n\bm \omega_k \bigr|\bigr|_{s(4)} + \bigl|\bigl| \sum_k \mu_k c_k^n \text{tr}\, \bm \omega_k \bigr|\bigr|_{-\frac{1}{4}, 4} \right) \\
& \leq C \bigl|\bigl| \sum_k \mu_k c_k^n\bm \omega_k \bigr|\bigr|_H .    
\end{align*}
We used that $ L^2(\Omega)  \hookrightarrow  L^{s(4)} (\Omega) $ and $ L^2(\partial \Omega)  \hookrightarrow W^{-\frac{1}{4}, 4} (\partial \Omega)$ in the two-dimensional setting.
Thanks to the definition of $\bu^n$ we have 
$$ || \bu^n ||_{1,4}^2 \leq C || L^n ||_H^2.$$
This completes the last part of the proof. Let us note that, if $\bu_0 \in V$, then we would obtain the result globally in time.

\qed \end{proof}

\begin{remark}
In contrast to the Dirichlet boundary data situation, we are not able to show that the velocity field belongs to $\bu \in L^2(0, T; W^{2, 2}(\Omega))$ using just the Galerkin approximation.
\end{remark}

Now, we will bootstrap the spatial regularity of solutions. We consider
the time derivative as a part of the right-hand side, and use the stationary 
theory mentioned in the previous section.

\begin{lemma}\label{thm:StokesEvolutionarFirstStep}
Let $1 < p \leq + \infty$, $1 < q < +\infty$, $\Omega \in \mathcal{C}^{1,1}$, $\bu_0 \in H$ and suppose that
\begin{align*}
& \bm F, \partial_t \bm F \in L^2(0, T; V^* ), \\
&\bm f \in L^p (0, T; L^{t( \min \{ q, 4\} )}(\Omega) ), \, \bm h \in L^p (0, T; W^{-\frac{1}{ \min \{ q, 4\}},  \min \{ q, 4\}}(\partial \Omega) ).
\end{align*}
Then the unique weak solution of \eqref{eq:StokesEvolutionary1}--\eqref{eq:StokesEvolutionary4} satisfies
$$ \bu \in L^p_{\text{loc}} (0, T; W^{1,  \min \{ q, 4\}}(\Omega)). $$
In particular, for $ p = +\infty$, $q > 2$, we obtain
$$ \bu \in L^\infty_{\text{loc}} (0, T; L^\infty(\Omega)). $$
Moreover, if 
\begin{align*}
&\bm f \in L^2_{\text{loc}} (0, T; L^2(\Omega) ), \,  \bm h \in L^2_{\text{loc}} (0, T; W^{\frac{1}{2}, 2}(\partial \Omega) ),
\end{align*}
then
$$ \bu \in L^2_{\text{loc}} (0, T; W^{2, 2}(\Omega)). $$
\end{lemma}

\begin{proof}
We want to move time derivatives in both main equations to the right-hand sides and apply \autoref{thm:StokesStationarWeak}. To do so, we need to verify 
\begin{align*}
\bm f - \partial_t \bu & \in L^p_{\text{loc}} (0, T; L^{t(\min \{ q, 4\})}(\Omega) ),\\
\beta \bm h - \beta \partial_t \bu & \in L^p_{\text{loc}} (0, T; W^{-\frac{1}{\min \{ q, 4\}}, \min \{ q, 4\}}(\partial \Omega) ).
\end{align*}
For our data $\bm f$, $\bm h$ it holds due to assumptions. Concerning the time derivatives we use \autoref{thm:StokesEvolutionarStrong} to get $\partial_t \bu \in L^\infty_{\text{loc}}(0, T; H)$. It implies two facts. First, $\partial_t \bu \in L^\infty_{\text{loc}}(0, T; L^2(\Omega)) \hookrightarrow L^\infty_{\text{loc}}(0, T; L^{t(\min \{ q, 4\})}(\Omega))$, which is due to $t(\min \{ q, 4\}) \leq 2$. Second, for the boundary term, we obtain $\beta \partial_t \bu \in L^\infty_{\text{loc}}(0, T; L^2(\partial \Omega)) \hookrightarrow L^\infty_{\text{loc}}(0, T; W^{-\frac{1}{\min \{ q, 4\}}, \min \{ q, 4\}}(\partial \Omega))$, because of Sobolev embedding. The case $p=\infty$ follows due to the embedding of $W^{1, q}$, $q > 2$, into $L^\infty$ in the two-dimensional case. The last part uses the fact that $\partial_t \bu \in L^2_{\text{loc}}(0, T; V)$ and \autoref{thm:StokesStationarHilbert}.
\qed \end{proof}

\begin{remark}
If we would assume $\bm F \in L^2(0, T; H)$ instead of both $\bm F$ and $\partial_t \bm F$ to be elements of $L^2(0, T; V^* )$, we could use \autoref{thm:StokesEvolutionarStrong}(ii) to obtain $ \bu \in L^p_{\text{loc}} (0, T; W^{1,  q}(\Omega)) $, $q > 2$, by interpolation.
\end{remark}

\begin{theorem}[$L^p-L^q$ regularity of evolutionary Stokes]\label{thm:StokesEvolutionarRegular}
Let $2 < p < + \infty$, $\Omega \in \mathcal{C}^{1,1}$ and suppose that
\begin{align*}
&\bm F \in L^2 (0, T; V^*) , \\
&\bm f \in L^p(0, T; L^p(\Omega) ), \, \bm h \in L^p(0, T; W^{1-\frac{1}{p}, p}(\partial \Omega) ).
\end{align*}
Moreover, let us assume that either
\begin{itemize}
\item[(i)] $\partial_t \bm F \in L^2(0, T; H)$, or
\item[(ii)] for some $2< \tilde{q} < 4$ there hold
\begin{align*}
&\partial_t \bm F, \partial_{tt} \bm F \in L^2 (0, T; V^*) , \\
& \partial_t \bm f \in L^p(0, T; L^{t(\tilde{q})}(\Omega) ), \, \partial_t \bm h \in L^p(0, T; W^{-\frac{1}{\tilde{q}}, \tilde{q}}(\partial \Omega) ).
\end{align*}
\end{itemize}
Then the unique weak solution of \eqref{eq:StokesEvolutionary1}--\eqref{eq:StokesEvolutionary4} satisfies, for a certain $ q > 2$,
\begin{align*}
\bu & \in L^\infty_{\text{loc}} (0, T; W^{1,q}(\Omega)), \\
\bu & \in L^p_{\text{loc}} (0, T; W^{2,q}(\Omega)), \\
\pi & \in L^p_{\text{loc}} (0, T; W^{1,q}(\Omega)).
\end{align*}
\end{theorem}

\begin{proof}
All assumptions of the previous lemma are satisfied. Therefore, we can interpolate between $L^2_{\text{loc}} (0, T; W^{2, 2}(\Omega))$ and $L^\infty_{\text{loc}}(0, T; W^{1, 2}(\Omega))$ to obtain that for a certain $q \in (2, p)$ there holds
$$\bu \in L^p_{\text{loc}}(0, T; W^{1, q}(\Omega)) . $$
Let us recall that \autoref{thm:StokesEvolutionarStrong} gives us
$$\partial_t \bu  \in L^\infty_{\text{loc}}(0, T; H) \cap L^2_{\text{loc}}(0, T; W^{1,2}(\Omega)),$$
and again, by a similar interpolation, we obtain
$$ \partial_t \bu \in L^p_{\text{loc}}(0, T; L^q (\Omega)),  $$
which gives us
$$\bm f - \partial_t \bu \in L^p_{\text{loc}}(0, T; L^q (\Omega)) . $$

Notice that we do not have enough regularity of the time derivative on the boundary to apply \autoref{thm:StokesStationarStrong}. We have only $\partial_t \bu \in L^2_{\text{loc}}(0, T; W^{ \frac{1}{2}, 2} (\partial \Omega)) \cap L^p_{\text{loc}}(0, T; W^{-\frac{1}{q}, q} (\partial \Omega))$, which is enough just for \autoref{thm:StokesStationarHilbert} or \autoref{thm:StokesStationarWeak}. 

To improve the time derivative we recall (as was argued during the proof of \autoref{thm:StokesEvolutionarStrong}) that the function $\bv = \partial_t \bu$ satisfies the same equation as $\bu$, just with $\partial_t \bm f$, $\partial_t \bm h$ instead of $\bm f$, $\bm h$. If there holds (i), we apply \autoref{thm:StokesEvolutionarStrong}(ii) to obtain 
$$ \bv  \in L^\infty_{\text{loc}} (0, T; W^{1,2}(\Omega)) \cap L^2_{\text{loc}}(0, T; W^{1,4}(\Omega)), $$
which interpolates into
$$ \partial_t \bu \in L^p_{\text{loc}}(0, T; W^{1, q} (\Omega)) . $$
If there holds (ii), we use \autoref{thm:StokesEvolutionarFirstStep} for $\bv$ and get
$$ \partial_t \bu = \bv \in L^p_{\text{loc}} (0, T; W^{1, \tilde{q} }(\Omega)). $$
Both $\bu$ and $\partial_t \bu$ belong into $L^p_{\text{loc}}(0, T; W^{1, q} (\Omega))$, for some $p, q > 2$, therefore
$$\bu \in L^\infty_{\text{loc}}(0, T; W^{1, q} (\Omega)).$$
In any case, we have $W^{1,q}(\Omega) \hookrightarrow  W^{ 1-\frac{1}{q}, q} (\partial \Omega) $. This means that for some $q > 2$ we have
$$  \beta \bm h - \beta \partial_t \bu  \in L^p_{\text{loc}}(0, T; W^{1-\frac{1}{q},q}(\partial \Omega)) . $$
This fact enables as us to invoke \autoref{thm:StokesStationarStrong} and get the final result.
\qed \end{proof}

In the following theorem, we prove the maximal-in-time regularity. 
The case $p=2$ is special, hence we formulate it separately.

\begin{theorem}[Maximal-in-time regularity of evolutionary  Stokes]\label{thm:StokesEvolutionarMaximal}
\noindent
\begin{itemize}
\item[(i)] Let $\Omega \in \mathcal{C}^{1,1}$ and assume
\begin{align*}
&\bm F,  \partial_t \bm F  \in L^2(0, T; V^*) , \\
&\bm f \in L^\infty(0, T; L^2(\Omega) ), \, \bm h \in L^\infty(0, T; W^{\frac{1}{2}, 2}(\partial \Omega) ).
\end{align*}
Moreover, let there hold either
\begin{align*}
 \partial_{tt} \bm F \in L^2(0, T; V^*)
\end{align*}
or 
\begin{align*}
\partial_t \bm F \in L^2(0, T; H) .
\end{align*}
Then the unique weak solution of \eqref{eq:StokesEvolutionary1}--\eqref{eq:StokesEvolutionary4} satisfies
\begin{align*}
\bu & \in L^\infty_{\text{loc}}(0, T; W^{2,2}(\Omega)) , \\
\pi & \in L^\infty_{\text{loc}}(0, T; W^{1,2}(\Omega)) .
\end{align*}
\item[(ii)] Let us now assume that $\Omega \in \mathcal{C}^{1,1}$ and for some $2 < q < 4$ there hold
\begin{align*}
&\bm F,  \partial_t \bm F ,  \partial_{tt} \bm F  \in L^2(0, T; V^*) ,   \\
&\bm f \in L^\infty(0, T; L^q(\Omega) ), \, \partial_t \bm f \in L^\infty(0, T; L^{t(q)}(\Omega) ) ,  \\
&\bm h \in L^\infty(0, T; W^{1-\frac{1}{q}, q}(\partial \Omega) ), \, \partial_t \bm h \in L^\infty(0, T; W^{-\frac{1}{q}, q}(\Omega) ).  
\end{align*}
Then we get
\begin{align*}
\bu & \in L^\infty_{\text{loc}} (0, T; W^{2,q}(\Omega)) , \\
\pi & \in L^\infty_{\text{loc}} (0, T; W^{1,q}(\Omega) ) .
\end{align*}
\end{itemize}

\end{theorem}

\begin{proof}
Because of \autoref{thm:StokesEvolutionarStrong} we have $\partial_t \bu \in L^\infty_{\text{loc}}(0, T; H)$, so
$$ \bm f -  \partial_t \bu \in L^\infty_{\text{loc}}(0, T; L^2(\Omega)) .$$
Considering the boundary term we have only
$$ \partial_t \bu \in L^\infty_{\text{loc}}(0, T; L^2(\partial \Omega)) \cap L^2_{\text{loc}}(0, T; W^{\frac{1}{2},2}(\partial \Omega)), $$
which is not enough for \autoref{thm:StokesStationarHilbert} to apply. To improve it, we apply either the first or the last part of \autoref{thm:StokesEvolutionarStrong} to the function $\bv=\partial_t \bu$. In any case, we obtain
$$ \bv \in L^\infty_{\text{loc}} (0, T; V), $$
and therefore
$$ \beta \partial_t \bu \in L^\infty_{\text{loc}} (0, T; W^{\frac{1}{2}, 2}(\partial \Omega)). $$
Thanks to our assumption on $\bm h$ we can use \autoref{thm:StokesStationarHilbert} and get the first part of our statement.

It remains to show (ii). As before, we already have $\partial_t \bu \in L^\infty_{\text{loc}} (0, T; V)$. Because of $W^{1, 2}(\Omega) \hookrightarrow L^q (\Omega)$, for any $ q  < +\infty$, we achieve
$$ \bm f -  \partial_t \bu \in L^\infty_{\text{loc}}(0, T; L^q(\Omega)) .  $$
To apply \autoref{thm:StokesStationarStrong} we need to get $\partial_t \bu \in L^\infty_{\text{loc}}(0, T; W^{1,q}(\Omega)) $, since then 
$$\beta \bm h - \beta \partial_t \bu \in L^\infty_{\text{loc}}(0, T; W^{1-\frac{1}{q}, q}(\partial \Omega) )$$
will be satisfied. Here, it is enough to apply \autoref{thm:StokesEvolutionarFirstStep} to $\bv = \partial_t \bu$, as in the previous theorem.
\qed \end{proof}

\section{Regularity for non-linear systems}

At this point we are prepared to focus on the more complicated systems, see \eqref{eq:Equation}--\eqref{eq:Initial}. First of all, we add the convective term to our equation in $\Omega$ and some non-linearity in $\bu$ into the equation on $\partial \Omega$. We will also cover the case of non-constant, yet
bounded viscosity. The whole procedure somehow mimics the method in \cite{Kaplicky}. 

\subsection{Existence of the solution}
As in the previous chapter, the starting point is again the Galerkin approximation, i.e. we look for the solution in the form
\begin{align*}
\bu^n & = \sum_{k=1}^n c_k^n(t) \bm \omega_k , \\
\bm S^n & = \bm S (\bm{Du} ^n) , \\
\bm s^n & = \bm s (\bu^n)
\end{align*}
that satisfies, for any $k = 1, \dots , n$, the system
\begin{align}
\begin{split}
(\partial_t \bu^n, \bm \omega_k)_H  + \int\limits_{\Omega} \bm S^n  : \bm{D\omega}_k & +  \int\limits_{\partial \Omega} \bm s^n \cdot \bm \omega_k \\
& = \left< \bm F, \bm \omega_k \right> - \int\limits_{\Omega} (\bu^n \cdot \nabla ) \bu^n \cdot \bm \omega_k  \label{eq:Galerkin}
\end{split}
\end{align}
together with the initial condition $c_k^n (0) = (\bu_0, \bm \omega_k)_H$.

\begin{theorem}[Existence of the weak solution for NS]\label{thm:NavierStokesExistence}
The problem \eqref{eq:Equation}--\eqref{eq:GrowthBoundary} with 
$$\bu_0 \in H, \, \Omega \in \mathcal{C}^{0, 1}, \, \bm F \in L^{2} (0, T; V^*),  $$
has a weak solution.
\end{theorem}

\begin{proof}
The proof is quite standard, see \cite{Maringova} or \cite{EM-dis} for more details. We multiply \eqref{eq:Galerkin} by $c_k^n$ and sum over $k = 1, \dots, n$ to obtain
\begin{align*}
\frac{1}{2} \cdot \frac{{\rm d}}{{\rm d}t} ||\bu^n ||_H^2  + \int\limits_{\Omega} \bm S^n : \bm{Du}^n +  \int\limits_{\partial \Omega} \bm s^n \cdot \bu^n = \left< \bm F, \bu^n \right>.
\end{align*}
Let us note that the convective term vanishes thanks to \eqref{eq:Div} and \eqref{eq:Impermeability}. Next, we use \eqref{eq:CoercivityBoundary} together with Korn's and Young's inequalities to get
\begin{align*}
 \frac{{\rm d}}{{\rm d}t} ||\bu^n ||_H^2  +  c|| \bu^n ||_{V}^2 \leq  C|| \bm F ||_{V^*}^{2}.
\end{align*}
Of course, we can also get the control of $\int\limits_{\partial \Omega} | \bu |^s$. This identity gives rise to uniform estimates in the form
\begin{align*}
\bu^n &\in L^\infty (0, T; H) \cap L^2 (0, T; V), \\
\partial_t \bu^n   & \in L^{2} (0, T; V^*),
\end{align*}
where the second one follows from the usual duality argument. Finally, we multiply \eqref{eq:Galerkin} by smooth function in time and proceed with the limit as $n \rightarrow +\infty$. Let us remark that in the non-linear terms we apply a standard monotonicity argument.
\qed \end{proof}

\begin{theorem}[Continuous dependence]\label{thm:NavierStokesUniqueness}
Let $\bu$, $\bv$ be weak solutions of \eqref{eq:Equation}--\eqref{eq:GrowthBoundary} with the same right-hand sides, then $\bm w := \bv  - \bu$ satisfies the inequality
\begin{align*}
 \frac{{\rm d}}{{\rm d}t} || \bm w||_H^2 + &  c || \bm {Dw} ||_2^2   +c \int\limits_{\partial \Omega} ( 1+ |\bv|^{s-2} + |\bu|^{s-2} ) |\bm w |^2  \leq C \big( 1 + || \bu ||_{V}^2 \big)
||\bm w ||_H^2 .
\end{align*}

Moreover, for any $t\in(0,T)$,
\begin{align} 
|| \bm w (t) ||_H^2 & \leq C || \bm w (0) ||_H^2 ,  \label{eq:ControlOfDifference1} \\
\int\limits_0^t ||\nabla \bm w ||_2^2  & \leq  C ||\bm w (0)||_H^2 \label{eq:ControlOfDifference2} . 
\end{align}
In particular, there exists at most one weak solution.
\end{theorem}

\begin{proof}
We take the difference of our equations and test it by the difference of two solutions, we use \eqref{eq:Interpolation2} in our estimates and obtain the desired inequality. 

Finally, we use Grönwall's inequality to show \eqref{eq:ControlOfDifference1}. By integration, and Korn's inequality, we can also control $\int\limits_0^t || {\bm w} ||_V^2 $. This implies the estimate $\int\limits_0^t || \bm w ||_V^2 \leq  C ||\bm w (0)||_H^2 $. Inequality \eqref{eq:ControlOfDifference2} then instantly follows and uniqueness is trivial.
\qed \end{proof}

\begin{remark}
The previous two theorems, together with the existence of the attractor, hold true also for $\bm S$ with more general growth and coercivity conditions. Additionally, no potential of $\bm S$ is actually needed. Moreover, we are able to do that also in the situation, where $\bm s$ is connected with $\bu$ via a so-called maximal monotone graph. For details, including the 3D setting, see \cite{PrZe22}.
\par
We note, however, that in the case of constitutive graphs, we are not 
able to obtain additional (time) regularity as in 
\autoref{thm:NavierStokesRegular}.
The problem of the attractor dimension is also largely open for this
important class of problems.
\end{remark}

\subsection{Regularity for NS system}

Let us now focus on the situation where $\bm S (\bm {Du}) 
= \nu \bm {Du}$, where $\nu>0$ is a constant; without loss of generality we
will temporarily set $\nu=1$. In other words, we want to learn how to deal with the non-linearity given by the presence of the convective term in $\Omega$ and the function $\bm s$ on its boundary.

\begin{theorem}[Regularity via Galerkin of NS]\label{thm:NavierStokesRegular}
Let us assume
\begin{align*}
&\Omega \in \mathcal{C}^{1,1}, \, \bu_0  \in V \cap W^{2, 2}(\Omega), \\
&\bm F, \partial_t \bm F  \in L^2(0, T; V^*) ,  \\
&\bm f(0) \in L^2(\Omega)  , \, \bm h(0) \in W^{\frac{1}{2},2}(\partial \Omega) .
\end{align*}
Then the unique weak solution of \eqref{eq:Equation}--\eqref{eq:DerivativeBoundary} has an additional regularity, namely
\begin{align*}
&\partial_t \bu \in L^\infty (0, T; H) \cap L^2 (0, T; V) , \\
&\bu \in L^\infty  (0, T; V) .
\end{align*}

Finally, the function $\bv := \partial_t \bu$ satisfies, for a.e. $t \in (0, T)$ and any $\bm \varphi \in V$, the equation
\begin{align*}
\langle \partial_t \bv, \bm \varphi \rangle  +    (\bv, \bm \varphi)_V   = \langle \tilde{\bm F}, \bm \varphi  \rangle ,
\end{align*}
where 
\begin{align*}
\tilde{\bm F} &= (\tilde{\bm f}, \tilde{\bm h}), \\
\tilde{\bm f} &= \partial_t \bm f  - (\bv  \cdot \nabla) \bu - (\bu \cdot \nabla) \bv , \\
\tilde{\bm h} &= \partial_t \bm h + \frac{1}{\beta} (\alpha - \bm s'(\bu)) \bv  .
\end{align*}

\end{theorem}

\begin{proof}

We proceed similarly as in \autoref{thm:StokesEvolutionarStrong}(i), i.e. we want to differentiate \eqref{eq:Galerkin} with respect to time. Let us note that it is basically the same procedure as in Theorem III.3.5 in \cite{Temam}.

Since $\bu_0 \in V \cap W^{2, 2}(\Omega)$, we can choose $\bu_0^n$ as the orthogonal projection in $V \cap W^{2, 2}(\Omega)$ of $\bu_0$ onto the space spanned by $\{ \bm \omega_k \}_{k=1}^n$. 
Therefore, 
$\bu_0^n \rightarrow \bu_0$ in $W^{2, 2}(\Omega)$ and $|| \bu_0^n ||_{2, 2} \leq || \bu_0 ||_{2,2}$. 
Next, we multiply \eqref{eq:Galerkin} by $(c_k^n)'(t)$, sum over $k=1, \dots , n$ and set $t = 0$ to obtain that $\partial_t \bu^n (0)$ is bounded in $H$.

Now, we take the time derivative of \eqref{eq:Galerkin} to get
\begin{align*}
(\partial_{tt} \bu^n, \bm \omega_k)_H & + \int\limits_{\Omega} \partial_t \bm{Du}^n  : \bm{D\omega}_k +  \int\limits_{\partial \Omega} \underbrace{ \bm{s}'(\bu^n)  \partial_t \bu^n }_{ \partial_t (\bm s(\bu^n))}   \cdot \bm \omega_k \\
& \quad  = \left< \partial_t \bm F, \bm \omega_k \right> - \int\limits_{\Omega} \left[ (\partial_t \bu^n  \cdot \nabla) \bu^n +  (\bu^n \cdot \nabla ) (\partial_t \bu^n) \right] \cdot \bm \omega_k.
\end{align*}
Further, let us multiply this equation by $(c_k^n)'(t)$ and sum over $k$'s to obtain
\begin{align*}
 & \frac{1}{2} \cdot \frac{ {\rm d}}{{\rm d}t} || \partial_t \bu^n ||_H^2   +  \int\limits_{\Omega} |\partial_t \bm{Du}^n |^2 +  \int\limits_{\partial \Omega}  \bm{s}'(\bu^n)  |\partial_t \bu^n |^2 \\
 & \quad = \left< \partial_t \bm F, \partial_t \bu^n \right> - \int\limits_{\Omega} (\partial_t \bu^n \cdot \nabla ) \bu^n \cdot \partial_t \bu^n +  (\bu^n \cdot \nabla ) (\partial_t \bu^n) \cdot \partial_t \bu^n.
\end{align*}
Thanks to
$$ \text{div}\, \partial_t \bu^n = 0 , \quad \partial_t \bu^n \cdot \bm n = 0 ,$$
we can simplify the equation to achieve
\begin{align*}
\frac{ {\rm d}}{{\rm d}t} || \partial_t \bu^n ||_H^2  & + 2\int\limits_{\Omega} |\partial_t \bm{Du}^n |^2 +  2\int\limits_{\partial \Omega}  \bm{s}'(\bu^n)  |\partial_t \bu^n |^2  \\
& \quad = 2\left< \partial_t \bm F, \partial_t \bu^n \right> - 2\int\limits_{\Omega} (\partial_t \bu^n \cdot \nabla ) \bu^n \cdot \partial_t \bu^n .
\end{align*}
Because of our assumptions, we can estimate the two last terms on the left-hand side in the following way
$$ 2\int\limits_{\Omega} |\partial_t \bm{Du}^n |^2 +  2\int\limits_{\partial \Omega}  \bm{s}'(\bu^n)  |\partial_t \bu^n |^2  \geq 2 c_5  || \partial_t \bu^n ||_V^2 .$$
Concerning the convective term we proceed just as in the standard Dirichlet setting. More specifically, we use Hölder's inequality, 
 interpolation \eqref{eq:Interpolation2} and Young's inequality to estimate
\begin{align*}
\left| \int\limits_{\Omega} (\partial_t \bu^n \cdot \nabla ) \bu^n \cdot \partial_t \bu^n \right|
    &\leq || \partial_t \bu^n ||_4^2 || \nabla \bu^n ||_2 
    \leq c  || \partial_t \bu^n ||_2 || \partial_t \bu^n ||_{1,2} || \bu^n ||_V \\
    &\leq \varepsilon || \partial_t \bu^n ||_V^2 + C || \partial_t \bu^n ||_H^2  || \bu^n ||_V^2 .
\end{align*}
Together, we  get the following inequality
\begin{align*}
\frac{ {\rm d}}{{\rm d}t} || \partial_t \bu^n ||_H^2  + c || \partial_t \bu^n ||_V^2  \leq C ||  \partial_t \bm F ||_{V^*}^2 + C || \bu^n ||_V^2 || \partial_t \bu^n ||_H^2 .
\end{align*}
Finally, we integrate over $(0, t)$ to obtain
\begin{align*}
|| \partial_t \bu^n (t)||_H^2  & + c_5   \int_0^t || \partial_t \bu^n ||_V^2  \\
& \leq || \partial_t \bu^n (0)||_H^2  +c \int_0^t  ||  \partial_t \bm F ||_{V^*}^2 + c \int_0^t  || \bu^n ||_{V}^2 || \partial_t \bu^n ||_H^2 
\end{align*}
and apply Grönwall's inequality to get
\begin{align*}
|| \partial_t \bu^n (t)||_H^2     \leq \left[ || \partial_t \bu^n (0)||_H^2  +c \int_0^t  ||  \partial_t \bm F ||_{V^*}^2 \right] \exp \left( c \int_0^t  || \bu^n ||_{V}^2 \right) . 
\end{align*}
As we already pointed out, everything on the right-hand side is bounded, and so the control of $\partial_t \bu $ in $L^\infty (0, T; H) \cap L^2 (0, T; V)$ follows. Because both $\bu$ and $\partial_t \bu$ belong into $L^2 (0, T; V)$ we obtain that $\bu \in L^\infty (0, T; V)$, which completes the first part of the proof.

To show the rest of the theorem we consider an arbitrary $ \psi \in \mathcal{C}^\infty_0 (0, T)$, multiply the weak formulation \eqref{eq:WF} by its derivative, and integrate over the whole time interval to achieve
\begin{align*}
\int\limits_0^T \left< \partial_t \bu, \bm  \varphi \right> \partial_t \psi & +  \int\limits_0^T \left( \int\limits_{\Omega} \bm {Du} : \bm {D \varphi}  + \int\limits_{\partial \Omega} \bm s (\bu) \cdot \bm  \varphi \right) \partial_t \psi \\
& = \int\limits_0^T \left< \bm F, \bm  \varphi \right> \partial_t \psi - \int\limits_0^T  \int\limits_{\Omega} (\bu \cdot \nabla  ) \bu \cdot \bm  \varphi \, \partial_t \psi  .
\end{align*}
Observe that $\partial_{tt} \bu \in L^2 (0, T; V^*) $, as follows from multiplicating the differentiated equation by $(c_k^n)''$. It means that we can use integration per partes in the first integral, in the other ones it is for free. Because $\bm \varphi$ does not depend on time and $\psi$ is compactly supported we get 
\begin{align*}
\int\limits_0^T \left< \partial_t \bv, \bm \varphi \right> \psi & + \int\limits_0^T \left( \int\limits_{\Omega} \bm {Dv} : \bm {D \varphi}  + \int\limits_{\partial \Omega} \bm s' (\bu) \bv \cdot \bm  \varphi \right)  \psi \\
& \qquad  = \int\limits_0^T \left< \partial_t \bm F, \bm \varphi \right> \psi - \int\limits_0^T  \int\limits_{\Omega} \left[ (\bv \cdot \nabla ) \bu + ( \bu \cdot \nabla)  \bv \right] \cdot \bm \varphi \, \psi  .
\end{align*}
This identity is satisfied for any smooth function, i.e. for a.e. $t \in (0, T)$ there holds
\begin{align*}
 \left< \partial_t \bv, \bm \varphi \right>  +    (\bv, \bm \varphi)_V   = \left< \left(\partial_t \bm f  - (\bv  \cdot \nabla) \bu - (\bu \cdot \nabla) \bv , \partial_t \bm h + \frac{1}{\beta} (\alpha - \bm s'(\bu)) \bv \right), \bm \varphi  \right> .
\end{align*}
\qed \end{proof}

\begin{remark}
Let us remark that we also can assume just $\bu_0 \in H$. The proof then works in the same way and we would obtain the same regularity as before, but locally in time.
\end{remark}

We will now state and prove the analogue to the last part of \autoref{thm:StokesEvolutionarStrong}. Nevertheless, we will not need it. The reason is that we can get a slightly better regularity with weaker assumptions on $\bm F$ using the stationary Stokes results.

\begin{lemma}
Let all the assumptions of the previous theorem hold and suppose that
$$ \bm F \in L^2(0, T; H) .  $$
Then the weak solution also satisfies
\begin{align*}
\bu & \in L^2(0, T; W^{1,4} (\Omega)).
\end{align*}
\end{lemma}

\begin{proof} We multiply \eqref{eq:Galerkin} by $(c_k^n)'(t)$ and sum over $k = 1, \dots, n$ to achieve
\begin{align}
|| \partial_t \bu^n ||_H^2 + \frac{1}{2} \cdot \frac{{\rm d}}{{\rm d}t} \int\limits_{\Omega} |\bm{Du}^n |^2 +  \int\limits_{\partial \Omega} \bm s^n \cdot \partial_t \bu^n  = \left< \bm F, \partial_t \bu^n \right> - \int\limits_{\Omega} (\bu^n \cdot \nabla) \bu^n  \cdot  \partial_t \bu^n . \label{eq:Galerkin1}
\end{align}
Simultaneously, we multiply\eqref{eq:Galerkin} by $\mu_k c_k^n(t)$ and sum over $k$'s again
\begin{align}
\frac{1}{2} \cdot \frac{{\rm d}}{{\rm d}t} || \bu^n ||_{V}^2 +  \int\limits_{\Omega} \bm{Du}^n : \bm D L^n +  \int\limits_{\partial \Omega} \bm s^n \cdot L^n  = \left< \bm F, L^n\right> - \int\limits_{\Omega} (\bu^n \cdot \nabla ) \bu^n  \cdot  L^n , \label{eq:Galerkin2}
\end{align}
where $L^n = \sum_{k=1}^n \mu_k c_k^n(t) \bm \omega_k $ as in \autoref{thm:StokesEvolutionarStrong}. By adding \eqref{eq:Galerkin1} and \eqref{eq:Galerkin2} we obtain 
\begin{align*}
& || \partial_t \bu^n ||_H^2  + \frac{1}{2} \cdot \frac{{\rm d}}{{\rm d}t} \left( || \bu^n ||_V^2  + \int\limits_{\Omega} |\bm{Du}^n |^2 \right) +   \int\limits_{\Omega} \bm{Du}^n : \bm D L^n  \\
&   = \left< \bm F, \partial_t \bu^n + L^n \right> - \int\limits_{\Omega} ( \bu^n \cdot \nabla) \bu^n  \cdot  ( \partial_t \bu^n + L^n) - \left(  \int\limits_{\partial \Omega} \bm s^n \cdot L^n +  \int\limits_{\partial \Omega} \bm s^n \cdot \partial_t \bu^n  \right) .
\end{align*}
As before, we know that
\begin{align*}
(L^n, L^n)_H  = || L^n ||_H^2 = (\bu^n, L^n)_{V} = \int\limits_{\Omega} \bm{Du}^n : \bm D L^n + \alpha \int\limits_{\partial \Omega} \bu^n \cdot L^n,
\end{align*}
and therefore
\begin{align*}
\int\limits_{\Omega} \bm{Du}^n : \bm D L^n = || L^n ||_H^2 - \alpha \int\limits_{\partial \Omega} \bu^n \cdot L^n.
\end{align*}
We rewrite the identity above in the following form
\begin{align*}
& ||  \partial_t \bu^n ||_H^2  + \frac{1}{2} \cdot \frac{{\rm d}}{{\rm d}t} \left( || \bu^n ||_{V}^2  + \int\limits_{\Omega} |\bm{Du}^n |^2 \right) +  || L^n ||_H^2  \\
&   = \left< \bm F, \partial_t \bu^n + L^n \right> - \int\limits_{\Omega} ( \bu^n \cdot \nabla) \bu^n  \cdot  ( \partial_t \bu^n + L^n)  \\
& \quad - \left(  \int\limits_{\partial \Omega} \bm s^n \cdot L^n +  \int\limits_{\partial \Omega} \bm s^n \cdot \partial_t \bu^n  \right) + \alpha \int\limits_{\partial \Omega} \bu^n \cdot L^n .
\end{align*}

Next, we integrate this equation over $(0, t)$, and thanks to Hölder's and Young's inequalities we get
\begin{align}
\begin{split}
\int\limits_0^t  || \partial_t \bu^n ||_H^2  & +  \left( || \bu^n ||_V^2  + \int\limits_{\Omega} |\bm{Du}^n |^2 \right)(t) +  \int\limits_0^t  || L^n ||_H^2  \label{eq:GalerkinNerovnost} \\
&   \leq \left( || \bu^n ||_V^2  + \int\limits_{\Omega} |\bm{Du}^n |^2 \right)(0)+  c \int \limits_0^t \int\limits_{\Omega} |\bu^n|^2 |\nabla \bu^n|^2  \\
& \quad + c \int \limits_0^t \left[  || \bm F ||_H^2  + \int\limits_{\partial \Omega} |\bu^n|^2 + \int\limits_{\partial \Omega} |\bm s^n|^2  \right].    
\end{split}
\end{align}
As we already saw in the proof of \autoref{thm:StokesEvolutionarStrong}, there holds
$$ || \bu^n ||_{1,4}^2 \leq C || L^n ||_H^2.$$
It gives us a way to deal with the convective term. Recall that because of \autoref{thm:NavierStokesRegular} we have $\{ \bu^n \}_n$ uniformly in $L^\infty (0, T; V)$ and we know that $W^{1, 2}(\Omega) \hookrightarrow L^q (\Omega)$ for any $q > 2$. Let us now consider any $\alpha \in (0, 1)$ and choose $\frac{1}{q} = \frac{1- \alpha}{4}$. For $\frac{1}{q} + \frac{1}{p} = \frac{1}{2}$ we get $p \in (2, 4)$, and therefore
\begin{align*}
 \int \limits_0^t \int\limits_{\Omega} |\bu^n|^2 |\nabla \bu^n|^2 
 & \leq   \int\limits_0^t  ||\bu^n||_q^2 ||\nabla \bu^n||_p^2 
  \leq C \int\limits_0^t  ||\nabla \bu^n||_p^2  \\
 & \leq C \int\limits_0^t  ||\nabla \bu^n||_2^{2\alpha} ||\nabla \bu^n||_4^{2(1-\alpha)}
  \leq C \int\limits_0^t  ||\nabla \bu^n||_4^{2(1-\alpha)} \\
  & \leq C \int\limits_0^t 1 +  \varepsilon \int\limits_0^t  ||\nabla \bu^n||_4^2 
  \leq CT +  \varepsilon \int\limits_0^t   || L^n ||_H^2 .
\end{align*}
We used Hölder's inequality and the uniform estimate for $\bu^n$, then the classical interpolation and the uniform estimate for $\nabla \bu^n$, lastly, Young's inequality (because $2(1-\alpha) < 2$) together with the estimate $ || \bu^n ||_{1,4}^2 \leq C || L^n ||_H^2$. Thus, from \eqref{eq:GalerkinNerovnost}, we finally obtain
\begin{align*}
\int\limits_0^t  || \partial_t \bu^n ||_H^2  &+   || \bu^n (t)||_{V}^2  +  \int\limits_0^t  || \bu^n ||_{1,4}^2  \\
& \leq  || \bu^n (0)||_{V}^2 +  c \int \limits_0^t \left[ 1+ || \bm F ||_H^2  + \int\limits_{\partial \Omega} |\bu^n|^2 + \int\limits_{\partial \Omega} |\bm s^n|^2  \right]  .    
\end{align*}
Thanks to the boundedness of the right-hand side we have the desired uniform control of $\bu^n$ in $L^2 (0, T; W^{1,4} (\Omega))$, which completes the proof.

\qed \end{proof}

\begin{remark}
In contrast with the Stokes problem we really need information about the time derivatives of our data (to control the convective term). Therefore, the previous lemma is not useful. As we will see, we are able to achieve $\bu \in L^2(0, T; W^{1,4} (\Omega))$ by use of the previous stationary theory with even weaker assumptions.
\end{remark}

Here, we will replicate \autoref{thm:StokesEvolutionarFirstStep} for our non-linear setting.

\begin{lemma}\label{thm:NavierStokesFirstStep}
Let all the assumptions of \autoref{thm:NavierStokesRegular} hold. Let us further assume that $\Omega \in \mathcal{C}^{1, 1}$ and for some $1< p \leq + \infty$ and $q \in (1, 4]$ there hold
\begin{align*}
\bm f & \in L^p(0, T; L^{t(q)}(\Omega) ), \, \bm h \in L^p(0, T; W^{-\frac{1}{q}, q}(\partial \Omega) ).
\end{align*}
Then the unique weak solution of \eqref{eq:Equation}--\eqref{eq:DerivativeBoundary} satisfies
\begin{align*}
\bu & \in L^p (0, T; W^{1, q} (\Omega)). 
\end{align*}
If the previous holds with $p = 2$ and, moreover,
\begin{align*}
\bm f & \in L^2(0, T; L^2(\Omega) ) , \, \bm h \in L^2(0, T; W^{\frac{1}{2}, 2}(\partial \Omega) ),
\end{align*}
then there also holds
\begin{align*}
\bu & \in L^2 (0, T; W^{2, 2} (\Omega)). 
\end{align*}
\end{lemma}

\begin{proof}
We wish to apply \autoref{thm:StokesStationarWeak}, i.e. we need to check that
\begin{align*}
\bm f - \partial_t \bu - (\bu \cdot \nabla) \bu & \in L^p(0, T; L^{t(q)}(\Omega) ), \\
\beta \bm h - \beta \partial_t \bu + \alpha \bu - \bm s(\bu)  & \in L^p(0, T; W^{-\frac{1}{q}, q}(\partial \Omega)). 
\end{align*}
For $\bm f$ and $\bm h$ it holds due to our assumptions and inclusions for $\partial_t \bu$ can be verified in the same way as in \autoref{thm:StokesEvolutionarFirstStep}. Just to recall, it follows from the fact that $\partial_t \bu \in L^\infty (0, T; H)$, which is true because of \autoref{thm:NavierStokesRegular}. Next, because $\bm s$ is Lipschitz and $\bu \in  L^\infty (0, T; W^{1, 2}(\Omega))$, we even have that $\alpha \bu - \bm s (\bu) \in L^\infty (0, T; W^{\frac{1}{2}, 2}(\partial \Omega))$. Finally, because  $\bu \in  L^\infty (0, T; L^q(\Omega))$ for any $q < + \infty$ and $\nabla \bu \in L^\infty (0, T; L^2(\Omega))$ we also get $(\bu \cdot \nabla) \bu \in L^\infty(0, T; L^{t(q)}(\Omega) )$, which finishes the first part of the proof.

To show the special case with $p = 2$ we want to use \autoref{thm:StokesStationarHilbert}, which means to verify
\begin{align*}
\bm f - \partial_t \bu - (\bu \cdot \nabla ) \bu & \in L^2(0, T; L^2(\Omega) ), \\
\beta \bm h  -  \beta \partial_t \bu + \alpha \bu - \bm s(\bu) & \in L^2(0, T; W^{\frac{1}{2}, 2}(\partial \Omega)). 
\end{align*} 
Up to the convective term is all clear, because of $\bu$, $\partial_t \bu \in L^2(0, T; W^{1, 2}(\Omega))$. To show that $(\bu \cdot \nabla ) \bu \in L^2(0, T; L^2(\Omega) )$ we recall that at this point we have $\bu \in L^p (0, T; W^{1, q} (\Omega)) \hookrightarrow L^p (0, T; L^\infty(\Omega))$ and $\nabla \bu \in L^\infty(0, T; L^2(\Omega))$, from which the conclusion follows by Hölder's inequality.
\qed \end{proof}

In correspondence with the previous section, we now develop $L^p-L^q$ regularity for finite $p$ and then also maximal time regularity, i.e. for $p = + \infty$.

\begin{theorem}[$L^p-L^q$ regularity of NS]\label{thm:NavierStokesp-q}
Let all the assumptions of \autoref{thm:NavierStokesRegular} hold. Let us further assume that $\bm s'$ is bounded and for some $2 < \sigma < 4 $ there holds 
\begin{align*}
\bm f & \in L^\infty(0, T; L^{t(\sigma)}(\Omega) ), \, \bm h \in L^\infty(0, T; W^{-\frac{1}{\sigma}, \sigma}(\partial \Omega) ).
\end{align*}
Let $ 2 < p < +\infty$ and
\begin{align*}
&\partial_t \bm F \in L^2 (0, T; H), \\
&\bm f \in L^p (0, T; L^p(\Omega)), \, \bm h \in L^p (0, T; W^{1-\frac{1}{p}, p}(\partial \Omega)), 
\end{align*}
Then the unique weak solution of \eqref{eq:Equation}--\eqref{eq:DerivativeBoundary} satisfies, for some $ q  > 2$, that
\begin{align*}
\bu & \in L^\infty_{\text{loc}} (0, T; W^{1, q}(\Omega)) , \\
\bu & \in L^p_{\text{loc}} (0, T; W^{2,q}(\Omega) ) , \\
\pi & \in L^p_{\text{loc}} (0, T; W^{1,q}(\Omega).
\end{align*}
\end{theorem}

\begin{proof}
Due to \autoref{thm:NavierStokesFirstStep} we immediately get
$$ \bu \in L^\infty_{\text{loc}} (0, T; W^{1, \sigma}(\Omega)) $$
and thus
$$ (\bu \cdot \nabla ) \bu  \in L^\infty_{\text{loc}} (0, T; L^{\sigma} (\Omega) ).$$ 
From \autoref{thm:NavierStokesRegular} we have $\partial_t \bu \in L^\infty (0, T; H) \cap L^2 (0, T; V)$ and it interpolates  into 
$$\partial_t \bu  \in L^p (0, T; L^{q} (\Omega) ), $$ 
where $q > 2$. Therefore,
$$ \bm f - \partial_t \bu -  (\bu \cdot \nabla ) \bu \in  L^p_{\text{loc}} (0, T; L^q (\Omega)),$$
which means that this (interior) term has the desired regularity to apply \autoref{thm:StokesStationarStrong}. It remains to show
$$\beta \bm h  - \beta \partial_t  \bu + \alpha \bu - \bm s (\bu ) \in L^p_{\text{loc}} (0, T; W^{1 - \frac{1}{q}, q}(\partial \Omega)).$$
The only problematic term is the time derivative, which needs to be improved.

To do so, we recall that thanks to \autoref{thm:NavierStokesRegular} there holds 
\begin{align*}
\langle \partial_t \bv, \bm \varphi \rangle  +    (\bv, \bm \varphi)_V   = \langle (\tilde{\bm f}, \tilde{\bm h}), \bm \varphi \rangle ,
\end{align*}
where
\begin{align*}
\tilde{\bm f} &= \partial_t \bm f  - (\bv \cdot \nabla ) \bu - (\bu \cdot \nabla ) \bv , \\
\tilde{\bm h} &= \partial_t \bm h + \frac{1}{\beta} (\alpha - \bm s'(\bu)) \bv  .
\end{align*}

Let us verify that $(\tilde{\bm f} , \tilde{\bm h}) \in L^2_{\text{loc}}  (0, T; H)$. In view of \autoref{lem:Decouple}, it is enough to show that
$\tilde{\bm f} \in L^2_{\text{loc}}  (0, T; L^2(\Omega))$ and $\tilde{\bm h} \in L^2_{\text{loc}} (0, T; L^2(\partial \Omega))$. For terms with time derivatives it follows from the assumptions.
 Because of the fact that $\sigma > 2$ we have $\bu \in L^\infty_{\text{loc}}(0, T; L^\infty(\Omega))$ and from \autoref{thm:NavierStokesRegular} follows $\nabla \bv \in L^2(0, T; L^2(\Omega))$. This information implies $(\bu \cdot \nabla ) \bv \in L^2_{\text{loc}}(0, T; L^2(\Omega))$. Next, because $\nabla \bu \in L^\infty_{\text{loc}} (0, T; L^\sigma(\Omega))$, $\sigma > 2$, and $\bv \in L^\infty (0, T; L^q(\Omega))$, for any $q < +\infty$, the Hölder's inequality gives $ (\bv \cdot \nabla ) \bu \in L^2 (0, T; L^2(\Omega))$. Therefore, $\tilde{\bm f} \in L^2_{\text{loc}} (0, T; L^2(\Omega))$. The integrability of boundary terms is now clear. Let us note that we implicitly used $\bm s(\bu) \cdot \bm n = 0$.

\iffalse

 To get it we can invoke \autoref{lem:Decouple}. Nevertheless, let us show that we can prove it without such a result. We need two things. First, $L^2$ integrability of both components, which we just did in the previous paragraph. Second, both components need to be compatible (see the definition of $H$). The couple $(\partial_t \bm f, \partial_t \bm h) \in L^2_{\text{loc}}  (0, T; H)$ by the assumption. The boundary term is also good; it is actually the trace of $W^{1, 2}(\Omega)$ function. Therefore, we have even $((\alpha - \bm s'(\bu)) \bv, \alpha - \bm s'(\bu)) \bv) \in L^2_{\text{loc}}  (0, T; V)$, see the Section 1.2. Thus, the remainings terms are
$$- (\bv \cdot \nabla ) \bu - (\bu \cdot \nabla ) \bv -(\alpha - \bm s'(\bu)) \bv \in L^2_{\text{loc}}  (0, T; L^2(\Omega)).$$
As we saw in the Section 1.2, this function can be completed by $\bm  0$ on the boundary and we then obtain 
$$ ( - (\bv \cdot \nabla ) \bu - (\bu \cdot \nabla ) \bv -(\alpha - \bm s'(\bu)) \bv, \bm 0 )  \in L^2_{\text{loc}}  (0, T; H). $$ 

\fi

Because the right-hand side $ (\tilde{\bm f}, \tilde{\bm h})$ of the evolutionary Stokes system belongs to $L^2_{\text{loc}}  (0, T; H)$, we can invoke \autoref{thm:StokesEvolutionarStrong}(ii) to achieve
$$\partial_t \bu = \bv \in L^\infty_{\text{loc}} (0, T; V) \cap  L^2_{\text{loc}} (0, T; W^{1,4} (\Omega)),  $$
which gives us for certain $q > 2$ that
$$ \partial_t  \bu \in L^p_{\text{loc}} (0, T; W^{1, q}(\Omega)), $$
by interpolation. This implies the desired regularity and \autoref{thm:StokesStationarStrong} gives us the last two inclusions in the assertion of the theorem. The first inclusion is a simple corollary of the fact that both $\bu$ and $\partial_t  \bu$ belong to $L^p_{\text{loc}} (0, T; W^{1, q}(\Omega))$ and $ p > 2$. Let us remark that the use of \autoref{thm:StokesEvolutionarStrong} above, gives us also information about the second time derivative, more specifically
$$ \partial_{tt}  \bu \in L^2_{\text{loc}} (0, T; H). $$
\qed \end{proof}

\begin{theorem}[Maximal regularity of NS]\label{thm:NavierStokesMaximal}
Let all the assumptions of \autoref{thm:NavierStokesRegular} hold and let us further assume that $\Omega \in \mathcal{C}^{1, 1}$ and $\bm s'$ is bounded.
\begin{itemize}
\item[(i)] Suppose that there hold
\begin{align*}
&\partial_t \bm F  \in L^2 (0, T; H), \\
&\bm f \in L^\infty (0, T; L^2(\Omega)) , \, \bm h \in L^\infty (0, T; W^{\frac{1}{2}, 2}(\partial \Omega)). 
\end{align*}
Then the unique weak solution of \eqref{eq:Equation}--\eqref{eq:DerivativeBoundary} satisfies
\begin{align*}
\bu & \in L^\infty_{\text{loc}} (0, T; W^{2,2}(\Omega) ) , \\
\pi & \in L^\infty_{\text{loc}} (0, T; W^{1,2}(\Omega).
\end{align*}
\item[(ii)] Suppose that $\bm s \in \mathcal{C}^2(\mathbb{R}^2)$, $\bm s''$ is bounded and for some $ 2 < p < +\infty$ there hold
\begin{align*}
&\partial_t \bm F  \in L^2 (0, T; H), \, \partial_{tt} \bm F \in L^2 (0, T;  V^*),  \\
&\bm f  \in L^\infty (0, T; L^p(\Omega)), \, \bm h \in L^\infty (0, T; W^{1-\frac{1}{p}, p}(\partial \Omega)), \\
&\partial_t \bm f \in L^\infty (0, T; L^{t(p)}(\Omega)), \, \partial_t \bm h \in L^\infty (0, T; W^{-\frac{1}{p}, p}(\partial \Omega)).
\end{align*}
Then we have for some $q > 2$ that
\begin{align*}
\bu & \in L^\infty_{\text{loc}} (0, T; W^{2,q}(\Omega) ) , \\
\pi & \in L^\infty_{\text{loc}} (0, T; W^{1,q}(\Omega) ) .
\end{align*}
\end{itemize}

\end{theorem}

\begin{proof}
Concerning the first part of the theorem we use \autoref{thm:NavierStokesRegular} and then \autoref{thm:NavierStokesFirstStep} to get
\begin{align*}
\bm f - \partial_t \bu - (\bu \cdot \nabla ) \bu  \in L^\infty_{\text{loc}} (0, T; L^2(\Omega)). 
\end{align*}
However, on the boundary, we have just 
$$\partial_t \bu \in L^\infty (0, T; L^2(\partial \Omega)) \cap L^2(0, T; W^{\frac{1}{2}, 2}(\partial \Omega)), $$ which is not enough. In the same fashion as in the previous theorem, we obtain $ \partial_t \bu \in L^\infty_{\text{loc}} (0, T; V),$ which gives us $ -\beta \partial_t \bu  \in L^\infty_{\text{loc}} (0, T; W^{\frac{1}{2}, 2}(\partial \Omega)) $. Therefore
\begin{align*}
\beta \bm h -\beta \partial_t \bu + \alpha \bu - \bm s(\bu)  \in L^\infty_{\text{loc}} (0, T; W^{\frac{1}{2}, 2}(\partial \Omega))
\end{align*}
and we can use \autoref{thm:StokesStationarHilbert} to finish the proof of (i).

To show (ii), we proceed as in \autoref{thm:NavierStokesp-q} to obtain
\begin{align*}
&\bu \in L^\infty_{\text{loc}} (0, T; W^{1, q}(\Omega))  , \\
&\bv \in L^\infty_{\text{loc}} (0, T; V) \cap  L^2_{\text{loc}} (0, T; W^{1,4} (\Omega)), \\
&\partial_{t}  \bv \in L^2_{\text{loc}} (0, T; H) ,
\end{align*}
for some $q > 2$. Together with $L^\infty_{\text{loc}} (0, T; V) \hookrightarrow L^\infty_{\text{loc}} (0, T; L^q(\Omega)) $ we see that
$$ \bm f - \partial_t \bu - (\bu \cdot \nabla ) \bu  \in  L^\infty_{\text{loc}} (0, T; L^q (\Omega)) $$
holds. This is exactly the regularity, of the ``interior'' term, which is needed to apply \autoref{thm:StokesStationarStrong}.  Hence, to use it, we need to achieve
$$ \partial_t  \bu \in L^\infty_{\text{loc}} (0, T; W^{1, q}(\Omega)). $$
Then,
$$ \beta \bm h - \beta \partial_t \bu +\alpha \bu - \bm s(\bu)   \in  L^\infty_{\text{loc}} (0, T; W^{1- \frac{1}{q}, q} (\partial \Omega))$$ 
will follow and \autoref{thm:StokesStationarStrong} gives the result.

To improve the time derivative we move $\partial_t \bv$ in 
\begin{align*}
\langle \partial_t \bv, \bm \varphi \rangle  +    (\bv, \bm \varphi)_V   = \langle (\tilde{\bm f}, \tilde{\bm h}), \bm \varphi \rangle ,
\end{align*}
where
\begin{align*}
\tilde{\bm f} &= \partial_t \bm f  - (\bv \cdot \nabla ) \bu - (\bu \cdot \nabla ) \bv , \\
\tilde{\bm h} &= \partial_t \bm h + \frac{1}{\beta} (\alpha - \bm s'(\bu)) \bv  ,
\end{align*}
to the right-hand side and use \autoref{thm:StokesStationarWeak}; it is actually nothing else than use of \autoref{thm:NavierStokesFirstStep} for $\bv$ instead of $\bu$. Therefore, we need to check
\begin{align*}
\partial_t \bm f - \partial_t \bv - (\bv \cdot \nabla ) \bu  - (\bu \cdot \nabla ) \bv & \in L^\infty_{\text{loc}}(0, T; L^{t(q)}(\Omega) ),  \\
\beta \partial_t \bm h - \beta \partial_t \bv + \alpha  \bv - \bm s'(\bu) \bv  & \in L^\infty_{\text{loc}}(0, T; W^{-\frac{1}{q}, q}(\partial \Omega)). 
\end{align*}
For most terms it is straightforward. Our data $(\partial_t \bm f, \partial_t \bm h)$ are improved in the assumptions of the theorem, the boundary term $\alpha  \bv - \bm s'(\bu) \bv $ is clear thanks to $\text{tr}\, \bv \in L^{\infty}_{\text{loc}}(0, T; W^{\frac{1}{2}, 2}(\partial \Omega))$ and boundedness of $\bm s'$. The nonlinear term  $(\bv \cdot \nabla ) \bu  + (\bu \cdot \nabla ) \bv$ belongs to $L^\infty(0, T; L^{t(q)}(\Omega) )$ because of the fact that both $\bu$ and $\bv$ belong to $L^\infty_{\text{loc}} (0, T; W^{1, 2}(\Omega))$. 
The only problem can occur in the time derivative $\partial_t \bv$; we need to improve it.

Let us again take a look at the equation
\begin{align*}
\langle \partial_t \bv, \bm \varphi \rangle  +    (\bv, \bm \varphi)_V   = \langle (\tilde{\bm f}, \tilde{\bm h}), \bm \varphi \rangle 
\end{align*}
and notice that, if we show $(\partial_t \tilde{\bm f}, \partial_t \tilde{\bm h}) \in L^2(0, T; V^*)$, then \autoref{thm:StokesEvolutionarStrong}(i) can be used and gives us 
$$ \partial_t \bv \in L^\infty_{\text{loc}}(0, T; H). $$
Of course, as we already saw in \autoref{thm:StokesEvolutionarFirstStep}, this regularity is enough to establish that both $ \partial_t \bv \in L^\infty_{\text{loc}}(0, T; L^{t(q)}(\Omega) )$  and  $\partial_t \bv \in L^\infty_{\text{loc}}(0, T; W^{-\frac{1}{q}, q}(\partial \Omega))$ are satisfied.

To finish the proof it remains to show $(\partial_t \tilde{\bm f}, \partial_t \tilde{\bm h}) \in L^2(0, T; V^*)$. First, we verify that
\begin{align*}
\partial_t \tilde{\bm f} &= \partial_{tt} \bm f -2(\bv \cdot \nabla ) \bv - ( \bu \cdot \nabla)( \partial_t \bv ) - (\partial_t \bv \cdot \nabla ) \bu \in L^2(0, T; (W^{1, 2}_{\sigma, \bm n} (\Omega))^* ) ,\\
\partial_t \tilde{\bm h}  &= \partial_{tt} \bm h + \frac{1}{\beta} (\alpha  - \bm s'(\bu)) \partial_t \bv  - \frac{1}{\beta} \bm s''(\bu) \bv \cdot \bv  \in  L^2(0, T; L^2(\partial \Omega))  .
\end{align*}
The worst terms are $( \bu \cdot \nabla)( \partial_t \bv )$ and $(\alpha  - \bm s'(\bu)) \partial_t \bv $. Nevertheless, our regularity of $\bu$ and $\bv$ is just enough to establish the required inclusions (together with the prescribed assumptions on $\partial_{tt} \bm f, \partial_{tt} \bm h$ and boundedness of $\bm s', \bm s''$). Second, we need the compatibility of the right-hand sides. As we already explained above the \autoref{lem:Decouple}, $(W^{1, 2}_{\sigma, \bm n} (\Omega))^* \times L^2(\partial \Omega) \hookrightarrow V^*$, and therefore $(\partial_t \tilde{\bm f}, \partial_t \tilde{\bm h}) \in L^2(0, T; V^*)$ indeed holds.

\iffalse

The most problematic are the boundary terms. Observe that $\bm s''(\bu) \bv \cdot \bv $ is the trace of Sobolev function and we thus can do the same trick as in the previous theorem. On the other hand, $(\alpha  - \bm s'(\bu)) \partial_t \bv$ is not trace of something. But, we know that $( (\alpha  - \bm s'(\bu)) \partial_t \bv, (\alpha  - \bm s'(\bu)) \partial_t \bv) \in L^2_{\text{loc}} (0, T; H)$, which means that this term can be handled too. Therefore, we need to add terms $\pm \bm s''(\bu) \bv \cdot \bv$ and $\pm (\alpha  - \bm s'(\bu)) \partial_t \bv$ to $\partial_t \tilde{\bm f}$; everything then works nicely and $(\partial_t \tilde{\bm f}, \partial_t \tilde{\bm h}) \in L^2(0, T; V^*)$ indeed holds and the proof is complete now.
\fi

\qed \end{proof}

\subsection{Regularity for systems with quadratic growth}

Here, we show the final regularity result, i.e. \autoref{thm:FirstMainTheorem}. Of course, its simple case $\bm S = \nu \bm {Du}$, with $\nu > 0$ constant, was treated in detail in \autoref{thm:NavierStokesMaximal}. Thus, from now on, we focus on the general case of the Cauchy stress $\bm S$ with a potential $U$, $U(0) = 0$, which is a $\mathcal{C}^3 (\mathbb{R}^+)$ function satisfying the estimates
\begin{align*} 
( \bm S ( \bm{D} ) - \bm S (\bm{E} ) ) : (\bm{D} - \bm{E}) & \geq c_1 | \bm{D} - \bm{E} |^2  ,  \\
 |\partial_{\bm D} U( |\bm D|^2)| =  |\bm S (\bm{D}) | & \leq c_2  |\bm D  | , \\
\partial_{\bm D}^2 U( |\bm D|^2) \bm{E} : \bm{E}  =  \partial_{\bm D} \bm S ( \bm{D} ) \bm{E} : \bm{E}  & \geq  c_1 | \bm E |^2 , \\
|\partial_{\bm D}^2 U( |\bm D|^2) | + |\partial_{\bm D}^3 U( |\bm D|^2) |  & \leq C,
\end{align*}
for all symmetrical $2\times 2$ matrices $\bm D, \bm E$.

\iffalse

\begin{theorem}\label{thm:QuadraticMaximal}
Let all assumptions of \autoref{thm:LinearMaximal} hold. Then there is $q > 2$ such that the unique weak solution of \eqref{eq:Equation}--\eqref{eq:GrowthBoundary} with $\bm S$ defined as above satisfies
\begin{align*}
\bu & \in L^\infty_{\text{loc}} (0, T; W^{2,q}(\Omega) ) , \\
\pi & \in L^\infty_{\text{loc}} (0, T; W^{1,q}(\Omega) ) .
\end{align*}
\end{theorem}

\fi

\begin{proof}[\textbf{Proof of \autoref{thm:FirstMainTheorem}.}]
We will not provide all the details; we only sketch how to modify previously developed methods, i.e. how to deal with the new non-linear term.

\underline{Step 1: Galerkin.} We start with repeating the proof \autoref{thm:NavierStokesRegular}. When differentiating the equation with respect to time we get the following expression coming from the elliptic term 
\begin{align*}
\int\limits_{\Omega} \partial_t \left( \bm S ( \bm{Du}^n ) \right)  : \bm{D}(\partial_t \bu^n) 
& = 
\int\limits_{\Omega} \partial_{\bm D} ( \bm S ( \bm{Du}^n ) ) \bm{D}(\partial_t \bu^n) : \bm{D}(\partial_t \bu^n) \\
&= 
\int\limits_{\Omega} \partial_{\bm D}^2 U( |\bm {Du}^n|^2 )   \bm{D}(\partial_t \bu^n) : \bm{D}(\partial_t \bu^n) \\
& \geq 
c_1 \int\limits_{\Omega}  | \bm{D}(\partial_t \bu^n) |^2 ,
\end{align*} 
where we used our assumption $\partial_{\bm D}^2 U( |\bm D|^2) \bm{E} : \bm{E}  \geq  c_1 | \bm E |^2$. We are thus able to control $L^2$-norm of $\bm{D}(\partial_t \bu^n)$ just as in the linear case. The rest of the proof is the same and we obtain
\begin{align*}
&\partial_t \bu \in L^\infty_{\text{loc}} (0, T; H) \cap L^2_{\text{loc}} (0, T; V) , \\
&\bu \in L^\infty_{\text{loc}} (0, T; V) .
\end{align*}
The corresponding problem for $\bv=\partial_t \bu$ will have the form
\begin{align*}
\langle \partial_t \bv, \bm \varphi \rangle  +  2\int\limits_{\Omega} U'( |\bm {Du}|^2)  \bm{Dv} : \bm{D\varphi} + \alpha \int\limits_{\partial\Omega} \bv \cdot \bm \varphi     = \langle \tilde{\bm F}, \bm \varphi  \rangle ,
\end{align*}
where 
\begin{align*}
\tilde{\bm F} &= (\tilde{\bm f}, \tilde{\bm g}), \\
\tilde{\bm f} &= \partial_t \bm f - 4 U''(|\bm {Du} |^2) \bm {Du} \bm {Du}  \bm {Dv}   - (\bv  \cdot \nabla) \bu - (\bu \cdot \nabla) \bv , \\
\tilde{\bm h} &= \partial_t \bm h + \frac{1}{\beta} (\alpha - \bm s'(\bu)) \bv  .
\end{align*}
Let us note that due to $|\partial_{\bm D}^2 U( |\bm D|^2) | \leq C$ we have the estimate 
$$4 |U''(|\bm {Du} |^2) \bm {Du} \bm {Du}  \bm {Dv} | \leq C | \bm {Dv}|.$$ It means that all terms are sufficiently integrable.

\underline{Step 2: Auxiliary result $\bu \in L^{\infty}_{\text{loc}} (0, T; W^{1, q} (\Omega)) $.} Here, we repeat the proof of \autoref{thm:NavierStokesFirstStep}. Recall that the leading elliptic term is given by $\bm S (\bm{Du}) = 2U'( |\bm{Du}|^2) \bm{Du}$. Therefore, there is no problem, because we can denote 
$$\bm A(t, x) := 2U'( |\bm{Du}|^2) $$ 
and use \autoref{thm:StokesStationarWeak}, together with the final remark in Section 2.2, to obtain
\begin{align*}
\bu & \in L^{\infty}_{\text{loc}} (0, T; W^{1, q} (\Omega)) 
\end{align*}
for some $q > 2$.

\underline{Step 3: First improvement of $\bv$.} Now, we replicate the method used in \autoref{thm:NavierStokesp-q}, i.e. we use \autoref{thm:StokesEvolutionarStrong}(ii) for the system
\begin{align*}
\langle \partial_t \bv, \bm \varphi \rangle  +  \int\limits_{\Omega} \bm{Dv} : \bm{D\varphi} + \alpha \int\limits_{\partial\Omega} \bv \cdot \bm \varphi     = \langle ( \tilde{\bm f} -  2U'( |\bm{Du}|^2) \bm{Du} + \bm{Dv}, \tilde{\bm h}), \bm \varphi  \rangle .
\end{align*}
The worst term in the first component is $\bm{Dv}$, but from the first step we already have $\bv \in L^2_{\text{loc}}(0, T; V)$, therefore, we achieve
$$( \tilde{\bm f} -  \bm A \bm{Dv} + \bm{Dv}, \tilde{\bm h}) \in L^2_{\text{loc}} (0, T; H).$$ \autoref{thm:StokesEvolutionarStrong} then gives us
\begin{align*}
&\bv \in L^\infty_{\text{loc}}(0, T; V) \cap L^2_{\text{loc}}(0, T; W^{1, 4}(\Omega)), \\
&\partial_t \bv \in L^2_{\text{loc}}(0, T; H).
\end{align*}
At this point we have $L^\infty_{\text{loc}} (0, T; L^q (\Omega))$ regularity of the interior term 
$$ \bm f - \partial_t \bu - (\bu \cdot \nabla ) \bu$$
and we need to improve $\partial_t \bu $ by one space derivative to control also the boundary term 
$$\beta \bm h - \beta \partial_t \bu + \alpha \bu - \bm s(\bu) $$
in the space $ L^\infty_{\text{loc}} (0, T; W^{1 - \frac{1}{q}, q} (\partial \Omega))$.

\underline{Step 4: Improvement of $\partial_t \bv$.} Here, just like in the proof of \autoref{thm:NavierStokesMaximal}, we show 
$$( \partial_t  \tilde{\bm f} -  \partial_t \left( 2U'( |\bm{Du}|^2) \bm{Du} \right)  + \bm D (\partial_t \bv), \partial_t \tilde{\bm h}) \in L^2_{\text{loc}} (0, T; V^*) .$$
We see that we have just enough information to guarantee it; let us just note that in $\partial_t  \tilde{\bm f}$ is contained the third derivative of $U$. Therefore, we use the second part of \autoref{thm:StokesEvolutionarStrong} and obtain
\begin{align*}
&\partial_t \bv \in L^{\infty}_{\text{loc}}(0, T; H).
\end{align*}

\underline{Step 5: Second improvement of $\bv$.} At this point, we move the time derivative of $\bv$, in the equation
\begin{align*}
\langle \partial_t \bv, \bm \varphi \rangle  +  \int\limits_{\Omega} \bm A  \bm{Dv} : \bm{D\varphi} + \alpha \int\limits_{\partial\Omega} \bv \cdot \bm \varphi     = \langle \tilde{\bm F}, \bm \varphi  \rangle ,
\end{align*}
to the right-hand side. Recall that $\bm A (t, x) = 2U'( |\bm{Du}|^2) $. As we already saw several times, $ \partial_t \bv \in L^\infty_{\text{loc}}(0, T; L^{t(q)}(\Omega) )$  and  $\partial_t \bv \in L^\infty_{\text{loc}}(0, T; W^{-\frac{1}{q}, q}(\partial \Omega))$ now hold. As above,
$$\left|4 U''(|\bm {Du} |^2) \bm {Du} \bm {Du}  \bm {Dv} \right| \leq C \left|\bm{Dv} \right| \in L^\infty_{\text{loc}}(0, T; L^{t(q)}(\Omega)).$$ 
Therefore, we can apply \autoref{thm:StokesStationarWeak} to this problem and get
\begin{align*}
&\bv \in L^{\infty}_{\text{loc}}(0, T; W^{1, q}(\Omega)).
\end{align*}

\underline{Step 6: Final conclusion.} Because \autoref{thm:StokesStationarStrong} holds also with the matrix $\bm A $ in the leading elliptic term, 
we can apply it to the system
\begin{align*}
\int\limits_{\Omega} \bm A \bm{Du} : \bm{D\varphi} + \alpha \int\limits_{\partial\Omega} \bu \cdot \bm \varphi     = \langle \bm F, \bm \varphi  \rangle  - \langle \partial_t \bu, \bm \varphi \rangle  .
\end{align*}
Thanks to $\partial_t \bu\in L^{\infty}_{\text{loc}}(0, T; W^{1, q}(\Omega))$ we get the desired regularity and the proof is complete.

\qed \end{proof}

\section{Dimension of the attractor}

We will now derive explicit estimates of the (fractal) dimension
of $\attr \subset H$, the global attractor to the system
\begin{align}	
\partial_t \bu - \diver \nu \bm{Du} + (\bu \cdot \nabla)\bu + \nabla\pi &= \bm f \qquad \textrm{in $(0, T) \times \Omega$}, \label{est-r1} \\
 \diver \bu & = 0 \qquad \textrm{in $(0, T) \times \Omega$}, \\				
\beta \partial_t \bu + \alpha \bu + [(\nu \bm{Du}) \bm n]_{\tau} &= \beta \bm h \qquad \textrm{on $(0, T) \times \partial \Omega$},\label{est-r2} \\
\quad \bu \cdot \bm n & = 0 \qquad \textrm{on $(0, T) \times \partial \Omega$}, \\
\bu(0) & = \bu_0 \qquad \textrm{in $\overline{\Omega}$} 
\end{align}
in terms of the data of the problem, that is to say,
the external forces $\bm f$ and $\bm h$, the constants $\nu$,
$\alpha$, $\beta$ and the characteristic length $\ell = \diam \Omega$. 
\par
We will now focus on the autonomous problem, i.e. the right-hand side
$\bm F = (\bm f, \bm h)$ is independent of time. Because of its uniqueness, the solution
semigroup $S(t):H \to H$, for $t\ge 0$, is well defined and continuous,
cf. \autoref{thm:NavierStokesUniqueness}. Existence of the global attractor
is also straightforward, see e.g. \cite[Theorem 1.2]{PrPr23}.
\par
We will apply the method of Lyapunov exponents, see \autoref{thm:Dimension}.
There are two main ingredients here. First, we need to verify the
differentiability of the solution operator. This crucially relies on
the regularity $\bu \in L^{\infty}(0,T;W^{2,q}(\Omega))$, for
some $q>2$, which is provided by \autoref{thm:FirstMainTheorem}. Note that as 
$\bm F $ does not depend on time, its assumptions
reduce to $\bm f \in L^p(\Omega)$, $\bm h \in W^{1-1/p,p}(\Bdry)$ 
for a certain $p>2$. Second, we want to estimate the trace of the
linearized operator.
\par
For the sake of simplicity, we only work with 
linear constitutive relations, but the whole procedure also works
if $\bm S$, $\bm s$ are non-linear functions with bounded derivatives.

\subsection{Differentiability of the solution operator}

Before we start, we need to make some notation and preparation. Two explicit a priori estimates are crucial here, namely:
\begin{align*}	
B_0 &= \sup_{\bu_0 \in \attr} \norm{\bu_0}{H}{},
	\\				
B_1 &= \sup_{\bu_0 \in \attr} \limsup_{t\to\infty}
	\frac1t \int\limits_0^t \norm{\bm{D u}}{\ldvaom}{2} \,{\rm d}\tau .
\end{align*}
The last integral is taken along solutions starting from $\bu_0$. We work with $\bm F \in H$ (and even better).
Testing the equation by $\bu$ in \eqref{eq:WF} and using \eqref{eq:Coercivity}, \eqref{eq:CoercivityBoundary} we obtain
\[
	\frac{1}{2} \cdot \ddt \norm{\bu}{H}{2} + c_1 \int\limits_{\Omega} |\bm{Du}|^2 + \alpha  c_3 \int\limits_{\partial \Omega} \left( |\bu|^2 +  |\bu|^s \right)
	= (\bm F,\bu)_H .
\] 
The following simple estimates will
be used repeatedly:
\begin{align}		\label{malfa}
	\norm{\bu}{V}{2} &\ge \calfa \norm{\bu}{W^{1,2}(\Omega)}{2},
		\qquad \calfa := \min\{1,\alpha\},	
\\		\label{Mbeta}
	\norm{\bu}{H}{2} &\le \cbeta \norm{\bu}{L^2(\Omega\times\partial \Omega)}{2},
		\qquad \cbeta := \max\{1,\beta\},    
\end{align}
where $L^2(\Omega\times \partial \Omega)=
L^2(\Omega) \times L^2(\partial \Omega)$ has the standard norm.

We can estimate
\begin{align*}
	 \ddt \norm{\bu}{H}{2} + 2c_1 \int\limits_{\Omega} |\bm{Du}|^2 + 2\alpha  c_3   \int\limits_{\partial \Omega} |\bu|^2
	& \leq 2|| \bm F ||_H || \bu ||_H  \\
	 \ddt \norm{\bu}{H}{2} + 2m || \bu ||_V^2
	& \leq 2|| \bm F ||_H || \bu ||_H  \\ 
     \ddt \norm{\bu}{H}{2}
	& \leq 2|| \bu ||_H \left( || \bm F ||_H  - m \frac{\calfa}{\cbeta} || \bu ||_H \right) ,
\end{align*}
where 
\begin{align} \label{eq:m}
   m :=   \min\{ c_1,  c_3  \}.
\end{align}
It follows that
\begin{align*}	
	B_0 &\le \frac{1}{m} \cdot \frac{\cbeta}{\calfa} \norm{\bm F}{H}{},
	\\			
	B_1 & \le \frac{B_0}{c_1} \norm{\bm F}{H}{} \le \frac{1}{m c_1} \cdot \frac{\cbeta}{\calfa} 
			\norm{\bm F}{H}{2} .
\end{align*}
Moreover, 
\begin{align*}
    \mathcal{B} := \overline{ \bigcup_{t \geq \tau} S(t) B(0, B_0) }
\end{align*}
is uniformly absorbing, positively invariant, and closed set for any fixed $\tau > 0$.

Now, we consider a formal linearization of our system \eqref{eq:Equation}--\eqref{eq:Initial}, i.e. 
\begin{align}
\partial_t \bU - \text{div}\,[\partial_{\bm D} \bm S (\bm{Du}) \bm{DU} ] + (\bU \cdot \nabla ) \bu + (\bu \cdot \nabla ) \bU + \nabla \sigma &= 0   , \label{eq:LinearEquation} \\
\text{div}\ \bU &= 0  \label{eq:LinearDiv} 
\end{align}
in $(0,T) \times \Omega$ together with
\begin{align}
\beta \partial_t \bU +   \bm s' (\bu)\bU  + [ (\partial_{\bm D} \bm S (\bm{Du}) \bm{DU}) \bm n ]_\tau &= 0 \quad \text{ on } (0,T) \times \partial \Omega , \label{eq:LinearDynamic} \\
\bU \cdot \bm n &= 0  \quad \text{ on } (0,T) \times \partial \Omega , \label{eq:LinearImpermeability}\\
\bU(\bm 0) &= \bv_0 - \bu_0  \quad \text{ in } \overline{\Omega} . \label{eq:LinearInitial}
\end{align}
Due to \eqref{eq:DerivativeBoundary}, \eqref{eq:StressDerivative} it clearly has a unique weak solution. We can prove the following.

\begin{theorem}
The solution operator $\mathcal{L}_t$ of \eqref{eq:LinearEquation}--\eqref{eq:LinearInitial} is a uniform quasidifferential to $S_t$ on $\mathcal{B}$, i.e. for any fixed $t > 0$ there holds
\begin{align}
|| \bv(t) - \bu(t) - \bU(t) ||_H = o ( ||\bv_0 - \bu_0 ||_H ), \quad ||\bv_0 - \bu_0 ||_H \rightarrow 0 , \label{eq:QuasiDiff} 
\end{align}
where $\bv, \bu$ solve \eqref{eq:Equation}--\eqref{eq:StressDerivative} with $\bv_0, \bu_0 \in \mathcal{B}$ respectively and $\bm U$ solves \eqref{eq:LinearEquation}--\eqref{eq:LinearInitial}.
\end{theorem}

\begin{proof}
We start with subtracting the equations for $\bm w := \bv - \bu$ and $\bU$ to obtain that 
\begin{align*}
\partial_t (\bm w - \bU) - \text{div}\, & \left[  \bm S (\bm {Dv}) - \bm S (\bm {Du})- \partial_{\bm D} \bm S (\bm{Du}) \bm{DU} \right] \\
& \qquad +(\bv \cdot \nabla ) \bv - (\bu \cdot \nabla ) \bu - (\bU \cdot \nabla )  \bu - (\bu \cdot \nabla ) \bU \\
& \qquad + \nabla \pi - \nabla \sigma = 0 .
\end{align*}
Next, we test it by $\bm w - \bU$, which leads to 
\begin{align}
\frac{1}{2} \cdot \frac{{\rm d}}{{\rm d}t} || \bm w - \bU||_H^2  + I_\Omega + I_{\partial \Omega} = J, \label{eq:Rovnost}
\end{align}
where
\begin{align*}
I_\Omega & := \int\limits_\Omega  \left[  \bm S (\bm {Dv}) - \bm S (\bm {Du})- \partial_{\bm D} \bm S (\bm{Du}) \bm{DU} \right] : \bm{D}(\bm w - \bU) , \\
I_{\partial \Omega} & := \int\limits_{\partial \Omega} \left[  \bm s (\bv) - \bm s (\bu)- \bm s' (\bu) \bU \right] : (\bm w - \bU)  , \\
J & := -\int\limits_\Omega  \left[ (\bv \cdot \nabla ) \bv -  (\bu \cdot \nabla ) \bu - (\bU \cdot \nabla ) \bu - ( \bu \cdot \nabla ) \bU  \right] \cdot (\bm w - \bU) .
\end{align*}

Now, we need to estimate these three integrals. Thanks to the differentiability of both $\bm S$ and $\bm s$ we can use the mean value theorem to find $\theta^1, \theta^2 \in [0, 1]$ such that 
\begin{align*}
I_\Omega 
&= \int\limits_\Omega  \left[  \partial_{\bm D} \bm S (\bm {Du} + \theta^1 \bm {Dw}) \bm {Dw} - \partial_{\bm D} \bm S (\bm{Du}) \bm{DU} \right] : \bm{D}(\bm w - \bU)  \\
&= I_\Omega^1 + I_\Omega^2 , \\
I_{\partial \Omega}
&= \int\limits_{\partial \Omega}  \left[   \bm s' (\bu + \theta^2 \bm w) \bm w - \bm s' (\bu) \bU \right] : (\bm w - \bU)  \\
&= I_{\partial \Omega}^1 + I_{\partial \Omega}^2 ,
\end{align*}
where
\begin{align*}  
I_\Omega^1 &= \int\limits_\Omega \partial_{\bm D} \bm S (\bm{Du}) \bm{D}(\bm w - \bU) : \bm{D}(\bm w - \bU), \\ 
I_\Omega^2 &= \int\limits_\Omega  \left[  \partial_{\bm D} \bm S (\bm {Du} + \theta^1 \bm {Dw}) \bm {Dw} - \partial_{\bm D} \bm S (\bm{Du}) \bm{Dw} \right] : \bm{D}(\bm w - \bU), \\
I_{\partial \Omega}^1 &= \int\limits_{\partial \Omega}  \bm s' (\bu) (\bm w - \bU) : (\bm w - \bU), \\ 
I_{\partial \Omega}^2 &= \int\limits_{\partial \Omega}  \left[  \bm s' (\bu + \theta^2 \bm w) \bm w - \bm s' (\bu) \bm w \right] : (\bm w - \bU).
\end{align*}
Because of \eqref{eq:StressDerivative} and \eqref{eq:DerivativeBoundary} we can estimate both $I_\Omega^1$ and $I_{\partial \Omega}^1$ as follows
$$ I_\Omega^1 + I_{\partial \Omega}^1 \geq c \int\limits_\Omega |\bm D (\bm w - \bU)|^2 + c \int\limits_{\partial \Omega} |\bm w - \bU|^2 \geq c || \bm w - \bU ||_{1,2}^2 ,  $$
where we also used Korn's inequality. Recall that derivatives of $\bm S$, $\bm s$ are actually Lipschitz, we can thus estimate the remaining two integrals in the following way
\begin{align*}
I_\Omega^2 
&
\leq \int\limits_\Omega |\bm {Dw} |^2 |\bm D (\bm w - \bU)| 
\leq c \int\limits_\Omega |\bm {Dw} |^4 + \varepsilon \int\limits_\Omega |\bm D (\bm w - \bU)|^2, \\
I_{\partial \Omega}^2 
& 
\leq \int\limits_{\partial \Omega} |\bm w |^2 |\bm w - \bU|  
\leq c \int\limits_{\partial \Omega} |\bm w |^4 + \varepsilon \int\limits_{\partial \Omega} |\bm w - \bU|^2 .
\end{align*}

Let us now rewrite the integral coming from the convective terms 
\begin{align*}
J 
&= \int\limits_\Omega  \left[ (\bu \cdot \nabla ) \bu -  (\bv \cdot \nabla ) \bv +  (\bU \cdot \nabla ) \bu + (\bu \cdot \nabla ) \bU  \right] \cdot (\bm w - \bU)  \\
&= \int\limits_\Omega  \left[ -(\bu \cdot \nabla ) \bm w + ( \bu \cdot \nabla ) \bU  -  ( \bm w \cdot \nabla ) \bv + ( \bU \cdot \nabla ) \bu  \right] \cdot (\bm w - \bU) \\
&= \int\limits_\Omega  \left[-(\bu \cdot \nabla) (\bm w - \bU)   -  (\bm w \cdot \nabla ) \bv + (\bU \cdot \nabla ) \bu  \right] \cdot (\bm w - \bU) \\
&= \int\limits_\Omega  \left[ -  (\bm w \cdot \nabla ) \bv + (\bU \cdot \nabla ) \bu  \right] \cdot (\bm w - \bU) \pm \int\limits_\Omega (\bm w \cdot \nabla ) \bu \cdot (\bm w - \bU) \\
&= \int\limits_\Omega  \left[  - (\bm w \cdot \nabla ) \bm w -  (\bm w \cdot \nabla ) \bu + (\bU \cdot \nabla ) \bu  \right] \cdot (\bm w - \bU) \\
&= -\int\limits_\Omega    (\bm w \cdot \nabla ) \bm w  \cdot (\bm w - \bU) - \int\limits_\Omega [ (\bm w -\bU) \cdot \nabla ] \bu   \cdot (\bm w - \bU) ,
\end{align*}
where from the first to second line we added $\pm \int\limits_\Omega \bu \nabla \bv \cdot (\bm w - \bU)$, from the third to fourth line the first term vanishes due to $\text{div}\, (\bm w - \bU) = 0$. Now, in the first integral, we use per partes and then Young's inequality gives us that
\begin{align*}
J &\leq \int\limits_\Omega  |\bm w|^2 | \nabla  (\bm w - \bU)| + \int\limits_\Omega |\bm w -\bU|^2 |\nabla \bu| \\
&\leq \varepsilon \int\limits_\Omega  | \nabla  (\bm w - \bU)|^2 + c\int\limits_\Omega  |\bm w|^4 + c\int\limits_\Omega |\bm w -\bU|^2 .
\end{align*}
Let us remark that here we have also used $\nabla \bu \in L^\infty (0, T; W^{1, 2}(\Omega))$.

Now, \eqref{eq:Rovnost}, together with the previous estimates, gives us the inequality
\begin{align*}
& \frac{1}{2} \cdot \frac{{\rm d}}{{\rm d}t} || \bm w - \bU||_H^2  + c || \bm w - \bU ||_{1,2}^2 \\
& \quad \leq  C \left( || \bm w ||_{L^4{(\Omega)}}^4 +  || \bm {Dw} ||_{L^4{(\Omega)}}^4 + || \bm w ||_{L^4{(\partial \Omega)}}^4  \right) +  C  \int\limits_\Omega |\bm w -\bU|^2     
\end{align*}
and due to Grönwall's inequality we obtain
$$  || (\bm w - \bU)(t)||_H^2   \leq 
 C e^{ct} \int\limits_0^t \left( || \bm w ||_{L^4{(\Omega)}}^4 +  || \bm {Dw} ||_{L^4{(\Omega)}}^4 + || \bm w ||_{L^4{(\partial \Omega)}}^4  \right)  . $$

In order to show \eqref{eq:QuasiDiff} we need to get
$$  \int\limits_0^t || \bm w ||_{L^4{(\Omega)}}^4 +  \int\limits_0^t|| \bm {Dw} ||_{L^4{(\Omega)}}^4 + \int\limits_0^t|| \bm w ||_{L^4{(\partial \Omega)}}^4 \leq C || \bm w_0||_H^{2+\delta}  $$
for some $\delta > 0$. Let us estimate integrals one by one. For the first one we have
\begin{align*}
 \int\limits_0^t ||\bm w||_4^4 \leq  \int\limits_0^t ||\bm w||_2^2 ||\bm w||_{1,2}^2 \leq C ||\bm w_0||_2^2  \int\limits_0^t  ||\bm w||_{1,2}^2 \leq C ||\bm w_0||_2^4 ,
\end{align*}
where we used interpolation \eqref{eq:Interpolation2} and estimates \eqref{eq:ControlOfDifference1}, \eqref{eq:ControlOfDifference2}. The next one is estimated as follows
\begin{align*}
\int\limits_0^t|| \bm {Dw} ||_{L^4{(\Omega)}}^4  & \leq \int\limits_0^t  ||\nabla \bm w ||_4^4  \leq   \int\limits_0^t  ||\nabla \bm w ||_2^{2+\alpha} ||\nabla \bm w ||_{1,q}^{2-\alpha} \\
& \leq \sup_{t\in (0, T) }  ||\nabla \bm w ||_{1, q}^{2-\alpha} \cdot \sup_{t\in (0, T) }  ||\nabla \bm w ||_2^{\alpha} \cdot \int\limits_0^t ||\nabla \bm w ||_2^2  \\
& \leq  C ||\bm w_0||_H^2 \cdot \sup_{t\in (0, T) }  ||\nabla \bm w ||_2^{\alpha} \leq C ||\bm w_0||_H^{2+\alpha / 2},
\end{align*}
where we used \eqref{eq:Interpolation3}, the fact $\bm w \in L^\infty (0, T; W^{2,q}(\Omega))$, estimates \eqref{eq:ControlOfDifference2} and the last inequality is due to the following estimate
\begin{align*}
\sup_{t\in (0, T) }  ||\nabla \bm w ||_2^{\alpha} & \leq \sup_{t\in (0, T) } \left(c || \bm w ||_2^{\alpha / 2} \cdot || \bm w ||_{2,2}^{ \alpha / 2} \right)  \\
& \leq C \cdot \sup_{t\in (0, T) }  || \bm w ||_2^{ \alpha / 2} \leq C ||\bm w_0||_H^{\alpha / 2}, 
\end{align*}
where \eqref{eq:Interpolation4}, $\bm w \in L^\infty (0, T; W^{2,2}(\Omega))$ and \eqref{eq:ControlOfDifference1} were needed. Concerning the last term we have
\begin{align*}
\int\limits_0^t|| \bm w ||_{L^4{(\partial \Omega)}}^4  & \leq   c\int\limits_0^t  || \bm w||_{1,4}^4 \leq C\int\limits_0^t \left( || \bm {Dv}||_4^4 + ||\text{tr}\, \bm w||_{L^2(\partial \Omega)}^4 \right) \\
& \leq  C ||\bm w_0||_H^{2+\alpha/2} + C \int\limits_0^t  || \bm w||_H^4 \leq  C ||\bm w_0||_H^{2+\alpha/2} + C ||\bm w_0||_H^{4} \\
& \leq  C ||\bm w_0||_H^{2+\alpha/2} ,
\end{align*}
where we used trace and Korn's inequalities, the previous estimate of the symmetrical gradient, and \eqref{eq:ControlOfDifference1}. By the choice $\delta = \alpha / 2$ we proved the desired estimate and the proof is complete.

\qed \end{proof}

\subsection{Trace estimates}
In view of suitable scaling (see Remark by the end of Appendix), we can assume that 
$\nu = \ell = 1$. In this setting, we have $m$ from \eqref{eq:m} equal to $1$, and therefore
\begin{align}	\label{est-B0}
	B_0 &= \sup_{\bu_0 \in \attr} \norm{\bu_0}{H}{} \le \frac{\cbeta}{\calfa} \norm{\bm F}{H}{},
	\\			\label{est-B1}
	B_1 &= \sup_{\bu_0 \in \attr} \limsup_{t\to\infty}
	\frac1t \int\limits_0^t \norm{\bm{D u}}{\ldvaom}{2} \,{\rm d}\tau  \le  \frac{\cbeta}{\calfa} 
			\norm{\bm F}{H}{2} .
\end{align}
We now need to estimate the $N$-trace of the linearized
equation, uniformly along the solutions on the attractor. More formally, writing the linearized equations \eqref{eq:LinearEquation}--\eqref{eq:LinearInitial} as 
\begin{equation} \label{lin-rce}
	\partial_t \bm U  = L(t,\bu_0) \bm U , 
\end{equation}
where $L(t,\bu_0)$ depends on a solution $\bu=\bu(t)$ with 
$\bu(0) = \bu_0 \in \attr$, we need to estimate
\begin{equation}	\label{N-trace}
	q(N) = \limsup_{t\to+\infty} \sup_{\bu_0 \in \attr}
			\sup_{\famn{\bm \varphi_j}} \frac1t \int_0^t
			\sum_{j=1}^N (L(\tau,\bu_0)\bm \varphi_j,\bm \varphi_j)\, {\rm d}\tau .
\end{equation}
The last supremum is taken over all families of functions
$\famn{\bm \varphi_j} \subset V$, which are orthonormal in $H$.
The quantity $q(N)$ provides an effective way to estimate the
global Lyapunov exponents, and a fortiori, of the attractor dimension, see \cite{Te97}.
In particular, if $q(N)<0$, then $\dimm{H} \attr \leq N$, cf.\  \autoref{thm:Dimension} 
in the Appendix.
\par
It follows that 
\begin{equation*}
- (L(\cdot,\bu_0)\bm \varphi_j,\bm \varphi_j) =
   \norm{\bm{D \varphi}_j}{\ldvaom}{2} 
+ \alpha \norm{\bm \varphi_j}{\ldvage}{2} 
- \intom (\bm \varphi_j \cdot \nabla) \bu \cdot \bm \varphi_j -
		 (\bu \cdot \nabla )\bm \varphi_j \cdot \bm \varphi_j
\end{equation*}
and thus
\begin{equation*}	
 \sum_{j=1}^N  (L(\cdot,\bu_0)\bm \varphi_j, \bm \varphi_j)
 \le - \calfa \sum_{j=1}^N \norm{\bm \varphi_j}{W^{1,2}(\Omega)}{2}
		+ \norm{\bm {Du}}{\ldvaom}{} \norm{\bm \rho}{\ldvaom}{}.
\end{equation*}
where $\bm \rho(x) = \sum_{j=1}^N |\bm \varphi_j(x)|^2$.
Invoking now \autoref{thm:LT2} below - recall that $\Omega$ 
has unit diameter, and  $\famn{\bm \varphi_j} $ are orthonormal in $H$,
hence suborthonormal in $L^2(\Omega)$ 
- we can estimate the second term as
\begin{align*}
	\norm{\bm{Du}}{\ldvaom}{} \norm{\bm \rho}{\ldvaom}{}
	&\le \frac{\calfa}{2\kappa} \norm{\bm \rho}{\ldvaom}{2}
		+ \frac{\kappa}{2\calfa}  \norm{\bm{Du}}{\ldvaom}{2}
\\
	&\le \frac{\calfa}{2}  \sum_{j=1}^N \norm{\bm \varphi_j}{W^{1,2}(\Omega)}{2}
		+ \frac{\kappa}{2\calfa}  \norm{\bm{Du}}{\ldvaom}{2} .
\end{align*}
This eventually yields
\begin{equation*}	
\sum_{j=1}^N  (L(\cdot,\bu_0)\bm \varphi_j,\bm \varphi_j)
	\le - \calfa  \sum_{j=1}^{N}  \norm{\bm \varphi_j}{W^{1,2}(\Omega)}{2}
		+ \calfa^{-1} \norm{\bm{Du}}{\ldvaom}{2} .
\end{equation*}
Also, by the min-max principle
\begin{equation*}
	\sum_{j=1}^{N}  \norm{\bm \varphi_j}{W^{1,2}(\Omega)}{2}
	\ge \sum_{j=1}^N \mu_j  \ge \cbeta^{-1} N^2 .
\end{equation*}
Here $\mu_j$ are eigenvalues of the corresponding Stokes
operator, see \autoref{thm:Basis}. The last inequality follows by the asymptotic
estimate $\mu_j \sim j$, see \autoref{thm:EA2d} below.
\par
Combining all the above with \eqref{est-B1}, we see that
\begin{equation*}
q(N) \le - \frac{\calfa}{\cbeta} N^2 + \calfa^{-1} B_1 \le 
		- \frac{\calfa}{\cbeta}  N^2 + \frac{\cbeta}{\calfa^2} \norm{\bm F}{H}{2}
\end{equation*}
and consequently, by \autoref{thm:Dimension} below, we obtain the desired estimate
\begin{equation}	\label{EST-1bd}
\dimm{H} \attr \le c_0 \frac{ \cbeta } {\calfa^{3/2}} \norm{\bm F}{H}{},
\end{equation}
where $c_0$ is some scale-invariant constant that only depends 
on the shape of $\Omega$. 

\subsection{Final evaluation of attractor dimension}
Recall that \eqref{EST-1bd} was actually obtained in terms of the
rescaled variables \eqref{backsc}, i.e., it should be written as
\begin{equation*}	
\dimm{\tilde H} \tilde{\attr} \le c_0 
        \frac{ M_{\tilde \beta}  } { {m_{\tilde \alpha}}^{3/2}} \norm{\tilde {\bm F}}{\tilde H}{},
\end{equation*}
But the rescaling does not affect attractor dimension. Observing also that
$\norm{\tilde {\bm F}}{\tilde H}{} = \ell^2 \nu^{-2} \norm{\bm F}{H}{}$, we
eventually come to 
\begin{equation}	\label{EST-final}
\dimm{H} {\attr} \le c_0 
        \frac{ \cbeta } { \calfa^{3/2}} \cdot \frac{\ell^2 \norm{\bm F}{H}{}}{\nu^2} \,,
\end{equation}
where (see \eqref{malfa}, \eqref{Mbeta} above)
\[
	\calfa = \min\{ 1 , \alpha \ell / \nu \}, 
	\qquad
	\cbeta = \max\{ 1 , \beta/\ell \} \,.
\]
%%%%% ZDE JSEM DISKUSI USEKL
Note these quantities are non-dimensional, as is the last term, which 
corresponds to the so-called Grashof number $G=|\Omega|\nu^{-2}\norm{\bm F}{H}{}$.
Hence, assuming that $\ell > \max\{ \beta, \nu/\alpha \}$, we recover the well-known estimate
$\dimm{L^2}{\attr} \le c_0 G$  for the Dirichlet boundary condition as a special (limiting) case.
%We can now distinguish several regimes:
%\begin{enumerate}
%\item If $\nu/\alpha < \ell < \beta$, 
%then $\calfa=1$, $\cbeta = \beta / \ell$, and
%\begin{equation*}
	%\dimm{H} \attr \le c_0 \frac{\beta \ell^2}{\nu^2} \norm{\bm F}{H}{} .
%\end{equation*}
%\item On the other hand, if $\ell$ is large so that
%$\ell > \max\{ \beta , \nu/\alpha \}$, we have $\calfa=\cbeta=1$
%and hence
%\begin{equation*}
	%\dimm{H} \attr \le c_0 \frac{\ell^3}{\nu^2} \norm{\bm F}{H}{} .
%\end{equation*}
%\item In case of $\ell$ small, i.e. $\ell < \min\{ \beta, \nu/\alpha \}$
%one gets
%\begin{equation*}
	%\dimm{H} \attr \le c_0 \frac{\beta \ell^{1/2}}{\alpha^{3/2} \nu^{1/2}}
		%\norm{\bm F}{H}{} .
%\end{equation*}
%\item Finally, $\beta < \ell < \nu/\alpha$ implies $\calfa=\alpha \ell / \nu$,
%$\cbeta=1$ and so
%\begin{equation*}
	%\dimm{H} \attr \le c_s \frac{\ell^{3/2}}{\alpha^{3/2} \nu^{1/2}}
	        %\norm{\bm F}{H}{} .
%\end{equation*}
%\end{enumerate}

%This is indeed analogous to the best available estimate
%for a 2D NSEs with Dirichlet boundary conditions, 
%see \cite{Te97}, where one has
%\[
	%\dimm{H} \attr \le c_0 \frac{|\Omega|}{\nu^2} 
		%\norm{\bm{f}}{2}{}.
%\]

\section{Appendix}

Here, for the reader's convenience, we collect some more or less well-known results. We start with a standard Sobolev embedding, a certain version of Korn's inequality, and some interpolations.

\begin{proposition}[Sobolev embedding]
Let $M$ be either Lipschitz $\Omega \subset \mathbb{R}^2$ or its boundary $\partial \Omega$. The space $W^{k, p}(M)$ is then continuously embedded into $W^{m, q}(M)$, provided 
$$ k \geq m, \, k - \frac{d}{p} \geq m - \frac{d}{q} \, , $$
where either $d=2$ if $M = \Omega$ or $d=1$ if  $M = \partial \Omega$.
\end{proposition}

\begin{proposition}[Sobolev traces]
Let $\Omega$ be a bounded Lipschitz domain. Then the range of the trace operator
is characterized by the equality
\newline
$\trace( W^{1,p}(\Omega) ) = W^{1-1/p,p}(\Bdry)$,
where $W^{1-1/p,p}(\Bdry)$ is the Sobolev-Slobodecki space, defined
via the norm
\[
	\left(
	\norm{ u }{L^p(\Bdry)}{p}
	+ \int_{\Bdry \times \Bdry} \frac{|u(x)-u(y)|^p}{|x-y|^{p+n-2}}\,{\rm d}x \, {\rm d}y
	\right)^{1/p}.
\]
\end{proposition}

\begin{proposition}[Korn's inequality]\label{thm:Korn}
Let $\Omega$ be a bounded Lipschitz domain and $r\in (1, \infty)$. Then there exists a constant $C > 0$, depending only on $\Omega$ and $r$, such that for all $\bu \in W^{1,r}(\Omega)$ that has $\textup{tr}\, \bu \in L^2(\partial \Omega)$, the following inequalities hold
$$ ||\bu ||_{1,r} \leq 
\begin{cases} 
C ( || \bm{Du} ||_r + ||\textup{tr}\,\bu ||_{L^2(\partial \Omega)})  \\
C ( || \bm{Du} ||_r + ||\bu ||_{L^2( \Omega)}) \end{cases} .$$
\end{proposition}

\begin{proof}
See Lemma 1.11 in \cite{BMR2007}.
\qed \end{proof}

\begin{proposition}[Interpolations]
Let $\Omega$ be a bounded Lipschitz domain in $\mathbb{R}^2$ and $q > 2$. Then there hold the following inequalities
\begin{align}
% || \bu ||_{2r'} & \leq c || \bu ||_2^{\frac{r-1}{r}} || \bu ||_{1,2}^{\frac{1}{r}}\, , \label{eq:Interpolation1} \\
|| \bu ||_4^4 & \leq c || \bu ||_2^2 || \bu ||_{1,2}^2 , \label{eq:Interpolation2} \\
|| \bu ||_4^4 & \leq c || \bu ||_2^{2+\gamma} || \bu ||_{1,q}^{2-\gamma}   \text{ with } \gamma = \frac{4(1-s)}{q}\, ,\, s \in \left(\frac{2}{q},1 \right), \label{eq:Interpolation3} \\
|| \nabla \bu ||_2^2 & \leq c || \bu ||_2 || \bu ||_{2,2}\, . \label{eq:Interpolation4} 
\end{align}
\end{proposition}

\begin{proof}
The first one is nothing else than the well-known Ladyzhenskaya's inequality, the other two can be found e.g. in \cite{KaPr}. They are based on the interpolation between $L^2$ and $L^\infty$ and the estimates $|| \bu ||_{\infty} \leq C || \bu ||_{s, q}$, $|| \bu ||_{s, q} \leq c || \bu ||_{q}^{1-s} || \bu ||_{1, q}^s $ for $s\in (0, 1)$.
\qed \end{proof}

Next, to establish an estimate of the dimension of the attractor we use the following result.

\begin{proposition}\label{thm:Dimension}
Let $\mathcal{A}$ be a compact set in a Hilbert space $H$, such that $\attr = S(t)\attr$
for some evolution operators $S(t)$. Let there exist uniform quasidifferentials $DS(t,u_0)$, 
which obey the equation of variations \eqref{lin-rce}, and let the corresponding global
Lyapunov exponents $q(N)$ be defined as in \eqref{N-trace}.
\par
Suppose further that $q(N) \le f(N)$, where $f(N)$ is a concave function, and $f(d)=0$
for some $d>0$. Then $\dimm{H}{\attr} \le d$.
\end{proposition}
\begin{proof}
See Theorem 2.1 and Corollary 2.2 in \cite{ChepIl01}.
\qed \end{proof}

%and there exists $DS(t,u_0)$ uniform quasidifferentials
%of $S(t)$. Let $DS(
%for any $u_0 \in \attr$.
%is uniformly quasidifferentiable on $\mathcal{A}$ and $S\mathcal{A} = \mathcal{A}$. Let $d = n + s$ for some $s \in (0, 1]$. Suppose that $D S(\bu)$ is norm-continuous with respect to $\bu \in \mathcal{A}$, i.e.
%$$ ||D S(\bu) - D S(\bu_0) ||_{\mathcal{L}(H, H)} \rightarrow 0 , \text{ as } ||\bu - \bu_0 ||_H  \rightarrow 0, \bu, \bu_0 \in \mathcal{A} . $$
%Suppose further that the quassidifferential $D S(\bu)$ contracts $d$-dimensional volumes uniformly for $\bu \in \mathcal{A}$, i.e.
%$$ \tilde{\omega}_d = \sup_{\bu \in \mathcal{A}} \omega_d (\bu) < 1. $$ 
%Then
%$$ \dimm{H} (\mathcal{A}) \leq d . $$

Further, we recall a generalized version of the celebrated
Lieb-Thirring inequality, following \cite{GMT88}.
A family of functions $\famn{\bm \varphi_j}$
is called suborthonormal in $L^2(\Omega)$, if for all $\famn{\xi_j} 
\subset \rr$ one has
\begin{equation}	\label{subon}
\sum_{i,j=1}^N
\xi_i  \xi_j (\bm \varphi_i,\bm \varphi_j)_{\ldvaom}
		\le \sum_{i=1}^N \xi_i^2 .
\end{equation}
A typical example are functions orthonormal in some larger space,
for example $V$ or $H$. Indeed, assuming that
\[
	(\bm \varphi_i,\bm \varphi_j)_{\ldvaom} + (\bm \varphi_i,\bm \varphi_j)_{\ldvage}
	= \delta_{ij}
\]
we readily obtain

\begin{align*}
\sum_{i,j=1}^N
\xi_i \xi_j (\bm \varphi_i,\bm \varphi_j)_{\ldvaom}
& = \sum_{i,j=1}^N \xi_i \xi_j \big( \delta_{ij} 
- (\bm \varphi_i,\bm \varphi_j)_{\ldvage} \big) \\
& = \sum_{i=1}^N \xi_i^2 - \norm{ \sum_{i=1}^N \xi_i \bm \varphi_i }{\ldvage}{2}
\le \sum_{i=1}^N  \xi_i^2 .
\end{align*}
The following version of the Lieb-Thirring inequality is used above.

\begin{proposition} \label{thm:LT2}
Let $\Omega \subset \rr^2$ be a bounded domain. Let $\famn{\bm \varphi_j} \subset W^{1,2}(\Omega)$ be suborthonormal in $\ldvaom$ and set
\begin{equation}    \label{def-rho}
 \rho(\bm x) = \sum_{j=1}^N |\bm \varphi_j(\bm x)|^2 .
\end{equation}
Then 
\begin{equation*}
\intom \rho^2 \le \kappa  \sum_{j=1}^N \left(
        \norm{\nabla \bm \varphi_j}{\ldvaom}{2}
                + \frac{1}{\diam{\Omega}}\norm{\bm \varphi_j}{\ldvaom}{2}
                \right),
\end{equation*}
where the constant $\kappa$ is independent of $N$.
\end{proposition}

\begin{proof}
Follows directly from \cite[Theorem 2.1]{GMT88}, with
$m=1$ and $n=k=p=2$.
\qed \end{proof}

Finally, we need to know something about the behavior of eigenvalues of our Stokes problem.

\begin{proposition} \label{thm:EA2d} Let $\mu_k$ be the sequence
of eigenvalues of the Stokes problem \eqref{eq:VlastniCisla}
with $\diam\Omega = 1$. Then $\mu_k \sim k$ as $k\to \infty$.
\end{proposition}

\begin{proof}
By the min-max principle, we can write
\begin{align*}
\mu_j & = \max \min \frac{ \intom |\nabla \bu |^2 + |\bu|^2 }
			{\intom |\bu|^2  + \beta \int_{\partial \Omega} |\bu|^2 } \\
	& \quad \ge \cbeta^{-1} 
		\max \min \frac{ \intom |\nabla \bu |^2 + |\bu|^2 }
			{\intom |\bu|^2  + \int_{\partial \Omega} |\bu|^2 }
	= \cbeta^{-1} \sigma_j .
\end{align*}
Here $\sigma_j$ are the eigenvalues corresponding to the Steklov
problem, which behave as $\sigma_j \sim j$ (recall that we
are in a bounded 2D domain, see \cite{BelFr05}).
\qed \end{proof}

%%%%%%%%%%%%%%%%%%%%%%%%%%%%%%%%%%%%%%%%
\begin{remark} By a suitable scaling, one can always assume:
\begin{enumerate}
\item $\nu=1$ and $\diam\Omega=1$, if $\Omega$ is bounded.
\item $\nu=\alpha=1$, if $\Omega$ is unbounded.
\end{enumerate}
\end{remark}
%%%%%%%%%%%%%%%%%%%%%%%%%%%%%%%%%%%%%%%%%%%%%%%%%%%%%%%%%%%%
\begin{proof}
Let (\ref{est-r1}--\ref{est-r2}) be given. Replacing
$\bu(\bm x,t)$ by $a \bu(\bm x/\ell,t/\tau)$ and $\pi(\bm x,t)$ by $a^2\pi(\bm x/\ell,t/\tau)$,
where $\ell = \diam\Omega$, one obtains
\begin{align*}	
\frac{a}{\tau}
\partial_t \bu - 
\frac{\nu a}{\ell^2} 
\diver \bm{Du} + \frac{a^2}{\ell} (\bu \cdot \nabla)\bu + 
\frac{a^2}{\ell} \nabla\pi &= \bm f, 
\, \diver \bu = 0 \quad \textrm{in $(0, T) \times \Omega$},
\\				
\frac{a \beta}{\tau \ell} \partial_t \bu + 
\frac{a \alpha}{\ell} \bu + 
\frac{a \nu }{\ell^2} [(\bm {Du})\bm n]_{\tau} &= \frac{\beta}{\ell} \bm h,
\, \bu \cdot \bm n = 0 \quad \textrm{on $(0, T) \times \partial \Omega$}.
\end{align*}
We now impose the relations
\begin{equation*}	
\frac{a}{\tau} = \frac{a^2}{\ell} = \frac{a \nu }{\ell^2} \, .
\end{equation*}
This implies that $\tau = \ell^2/\nu$, $a=\nu/\ell$, in terms
of given $\ell$, $\nu>0$. Dividing both equations by 
$a/\tau = \nu^2/\ell^3$, we come to
\begin{align*}   
\partial_t \bu - \diver \bm{Du} + (\bu \cdot \nabla)\bu + \nabla\pi &= \tilde{\bm f},
\quad \diver \bu = 0 \qquad \textrm{in $(0, T) \times\Omega$},
\\             
\tilde \beta \partial_t \bu + \tilde \alpha \bu + [(\bm{Du})\bm n]_{\tau} &= \tilde \beta \tilde{\bm h},
\quad \bu \cdot \bm n = 0 \qquad \textrm{on $(0, T) \times \partial \Omega$},
\end{align*}
where
\begin{equation} \label{backsc}
\tilde \alpha = \frac{\alpha \ell}{ \nu}\, , \quad
\tilde \beta = \frac{\beta}{\ell}\, , \quad
\tilde{\bm f} = \frac{\ell^3}{\nu^2} \bm f , \quad
\tilde{\bm h} = \frac{\ell^3}{ \nu^2} \bm h  .
\end{equation}
In the case of $\Omega$ unbounded, we are also free to choose
$\ell= \nu/\alpha$ so that $\tilde\alpha=1$ and 
\begin{equation*}	
\tilde \beta = \frac{\alpha \beta}{\nu}\, , \quad
\tilde{\bm f}  = \frac{\nu}{\alpha^3} \bm f , \quad
\tilde{\bm h} = \frac{\nu}{\alpha^3} \bm h .
\end{equation*}
\qed \end{proof}

\bibliography{bibliography}

\begin{thebibliography}{10}

\bibitem{Maringova}
A.~Abbatiello, M.~Bul\'{\i}\v{c}ek, and E.~Maringov\'{a}.
\newblock On the dynamic slip boundary condition for {N}avier--{S}tokes-like problems.
\newblock {\em Mathematical Models and Methods in Applied Sciences}, 31(11):2165--2212, 2021.

\bibitem{AACG21}
P.~Acevedo, C.~Amrouche, C.~Conca, and A.~Ghosh.
\newblock Stokes and {N}avier-{S}tokes equations with {N}avier boundary conditions.
\newblock {\em J. Differential Equations}, 285:258--320, 2021.

\bibitem{BMR2007}
M.~Bul\'{\i}\v{c}ek, J.~M\'{a}lek, and K.~R. Rajagopal.
\newblock Navier's slip and evolutionary {N}avier-{S}tokes-like systems with pressure and shear-rate dependent viscosity.
\newblock {\em Indiana Univ. Math. J.}, 56(1):51--85, 2007.

\bibitem{ChepIl01}
V.~V. Chepyzhov and A.~A. Ilyin.
\newblock A note on the fractal dimension of attractors of dissipative dynamical systems.
\newblock {\em Nonlinear Anal., Theory Methods Appl., Ser. A, Theory Methods}, 44(6):811--819, 2001.

\bibitem{CF:1988}
P.~Constantin and C.~Foias.
\newblock {\em Navier-Stokes Equations}.
\newblock University of Chicago Press, Chicago, 1988.

\bibitem{FN09}
E.~Feireisl and A.~Novotn\'{y}.
\newblock {\em Singular limits in thermodynamics of viscous fluids}.
\newblock Advances in Mathematical Fluid Mechanics. Birkh\"{a}user Verlag, Basel, 2009.

\bibitem{Galdi11}
G.~P. Galdi.
\newblock {\em An introduction to the mathematical theory of the {N}avier-{S}tokes equations}.
\newblock Springer Monographs in Mathematics. Springer, New York, second edition, 2011.
\newblock Steady-state problems.

\bibitem{GMT88}
J.-M. Ghidaglia, M.~Marion, and R.~Temam.
\newblock Generalization of the {S}obolev-{L}ieb-{T}hirring inequalities and applications to the dimension of attractors.
\newblock {\em Differential Integral Equations}, 1(1):1--21, 1988.

\bibitem{AG-dis}
A.~Ghosh.
\newblock {\em Navier-Stokes equations with Navier boundary condition}.
\newblock PhD thesis, Université de Pau et des Pays de l’Adour and Universidad del País Vasco, Pau, 2018.

\bibitem{IPZ:2016}
A.~Ilyin, K.~Patni, and S.~Zelik.
\newblock Upper bounds for the attractor dimension of damped {N}avier-{S}tokes equations in {$\Bbb{R}^2$}.
\newblock {\em Discrete Contin. Dyn. Syst.}, 36(4):2085--2102, 2016.

\bibitem{ILZE:21}
A.~Ilyin and S.~Zelik.
\newblock Sharp dimension estimates of the attractor of the damped 2{D} {E}uler-{B}ardina equations.
\newblock In {\em Partial differential equations, spectral theory, and mathematical physics---the {A}ri {L}aptev anniversary volume}, EMS Ser. Congr. Rep., pages 209--229. EMS Press, Berlin, [2021] \copyright 2021.

\bibitem{ilyin93}
A.~A. Ilyin.
\newblock Partly dissipative semigroups generated by the {Navier}-{Stokes} system on two-dimensional manifolds, and their attractors.
\newblock {\em Russ. Acad. Sci., Sb., Math.}, 78(1):55--88, 1993.
\newblock (English translation in \textit{Sbornik Mathematics} \textbf{78}(1), 47-76 (1994)).

\bibitem{Kaplicky}
P.~Kaplick\'{y}.
\newblock Regularity of flows of a non-{N}ewtonian fluid subject to {D}irichlet boundary conditions.
\newblock {\em Journal for Analysis and its Applications}, 24(3):467--486, 2005.

\bibitem{KaPr}
P.~Kaplick\'{y} and D.~Pra\v{z}ák.
\newblock Differentiability of the solution operator and the dimension of the attractor for certain power-law fluids.
\newblock {\em Journal of Mathematical Analysis and Applications}, 326(1):75--87, 2007.

\bibitem{EM-dis}
E.~Maringov\'{a}.
\newblock {\em Mathematical Analysis of Models Arising in Continuum Mechanics with Implicitly Given Rheology and Boundary Conditions}.
\newblock PhD thesis, Faculty of Mathematics and Physics, Charles University, Prague, 2019.

\bibitem{PrPr23}
D.~Pra\v{z}ák and B.~Priyasad.
\newblock The existence and dimension of the attractor for a 3{D} flow of a non-{N}ewtonian fluid subject to dynamic boundary conditions.
\newblock {\em Applicable Analysis}, 0(0):1--18, 2023.

\bibitem{PrZe22}
D.~Pra\v{z}ák and M.~Zelina.
\newblock On the uniqueness of the solution and finite-dimensional attractors for the 3{D} flow with dynamic slip boundary condition.
\newblock (submitted).

\bibitem{Rob01}
J.~C. Robinson.
\newblock {\em Infinite-dimensional dynamical systems}.
\newblock Cambridge Texts in Applied Mathematics. Cambridge University Press, Cambridge, 2001.
\newblock An introduction to dissipative parabolic PDEs and the theory of global attractors.

\bibitem{Rob11}
J.~C. Robinson.
\newblock {\em Dimensions, Embeddings, and Attractors}.
\newblock Cambridge University Press, 2011.

\bibitem{Temam}
R.~Temam.
\newblock {\em Navier-{S}tokes {E}quations: {T}heory and {N}umerical {A}nalysis}.
\newblock North-Holland, 1979.

\bibitem{Te97}
R.~Temam.
\newblock {\em Infinite-dimensional dynamical systems in mechanics and physics}, volume~68 of {\em Applied Mathematical Sciences}.
\newblock Springer-Verlag, New York, second edition, 1997.

\bibitem{BelFr05}
J.~von Below and G.~Fran\c{c}ois.
\newblock Spectral asymptotics for the {L}aplacian under an eigenvalue dependent boundary condition.
\newblock {\em Bull. Belg. Math. Soc. Simon Stevin}, 12(4):505--519, 2005.

\bibitem{ziane98}
M.~Ziane.
\newblock On the two-dimensional {Navier}-{Stokes} equations with the free boundary condition.
\newblock {\em Appl. Math. Optim.}, 38(1):1--19, 1998.

\end{thebibliography}
\bibliographystyle{plain} 

\end{document}